\documentclass[a4paper]{amsart}

\usepackage{amsthm, amssymb, amsmath, amsfonts}
\usepackage{url}

\newtheorem{thm}{Theorem}[section]
\newtheorem{prop}[thm]{Proposition}
\newtheorem{lem}[thm]{Lemma}

\renewcommand{\le}{\leqslant}
\renewcommand{\ge}{\geqslant}
\renewcommand{\subset}{\subseteq}
\newcommand{\mcD}{\mathcal{D}}

\newcommand{\E}{\mathbb{E}}

\newcommand{\EEt}{\mathbf{E}^\tau}

\newcommand{\hX}{\hat{X}}

\newcommand{\htau}{\hat{\tau}}

\newcommand{\N}{\mathbb{N}}

\newcommand{\1}{\mathbf{1}}
\newcommand{\R}{\mathbb{R}}

\newcommand{\Z}{\mathbb{Z}}
\renewcommand{\P}{\mathbb{P}}
\newcommand{\PP}{\mathbf{P}}
\newcommand{\PPt}{\mathbf{P}^\tau}

\newcommand{\ov}{\overline}
\newcommand{\td}{\tilde}
\newcommand{\eps}{\varepsilon}
\def\d{{\mathrm{d}}}

\newcommand{\Hde}{H_\delta^{(\eps)}}
\newcommand{\Hdek}{H_\delta^{(\eps_k)}}
\newcommand{\Lde}{{L}_\delta^{(\eps)}}
\newcommand{\Hdo}{H_\delta}

\newcommand{\var}{\mathbb{V}\mathrm{ar}}
\newcommand{\de}{_\delta^{(\eps)}}
\newcommand{\dek}{_\delta^{(\eps_k)}}
\newcommand{\ek}{^{(\eps_k)}}
\newcommand{\e}{^{(\eps)}}

\newcommand{\mfk}{\mathfrak}

\title{Scaling limit of the random walk among random traps on $\Z^d$}
\author{Jean-Christophe Mourrat}

\address{Universit\'e de Provence, CMI, 39 rue Joliot Curie, 13013 Marseille, France ; PUC de Chile, Facultad de Matem\'aticas, Vicu\~na Mackenna 4860, Macul, Santiago, Chile.}

\begin{document}
\begin{abstract}
Attributing a positive value $\tau_x$ to each $x \in \Z^d$, we investigate a nearest-neighbour random walk which is reversible for the measure with weights $(\tau_x)$, often known as ``Bouchaud's trap model''. We assume that these weights are independent, identically distributed and non-integrable random variables (with polynomial tail), and that $d \ge 5$. We obtain the quenched subdiffusive scaling limit of the model, the limit being the fractional kinetics process. We begin our proof by expressing the random walk as a time change of a random walk among random conductances. We then focus on proving that the time change converges, under the annealed measure, to a stable subordinator. This is achieved using previous results concerning the mixing properties of the environment viewed by the time-changed random walk.
\end{abstract}
\maketitle
\section{Introduction}
In this paper, we consider a trap model, known as \emph{Bouchaud's trap model} or also the \emph{random walk among random traps}, evolving on the graph $\Z^d$. In this model, every site $x \in \Z^d$ represents a trap of a certain depth $\tau_x > 0$, and the dynamics is chosen in order to make the measure with weights $(\tau_x)_{x \in \Z^d}$ reversible. More precisely, for a fixed $a \in [0,1]$ and $\tau = (\tau_x)_{x \in \Z^d}$, we consider the continuous time Markov chain $(X_t)_{t \ge 0}$ whose jump rate from a site $x$ to a neighbour $y$ is 
\begin{equation}
\label{jumprate}
\frac{(\tau_y)^a}{(\tau_x)^{1-a}}.
\end{equation}
We write $\PPt_x$ for the law of this process starting from $x \in \Z^d$, $\EEt_x$ for its associated expectation. 
The environment $\tau$ is chosen according to some probability law $\P$ (with corresponding expectation $\E$). We focus here on the case when $(\tau_x)_{x \in \Z^d}$ are independent and identically distributed, and in the regime where there is some $\alpha \in (0,1)$ such that
\begin{equation}
\label{regvar}
\P[\tau_0 \ge y] \sim \frac{1}{y^\alpha} \qquad (y \to +\infty).
\end{equation}
In particular, the expectation of the depth of a trap is infinite. We also assume that $\tau_0 \ge 1$. For $\eps > 0$, we define the rescaled process $X\e(t) = \sqrt{\eps} X_{\eps^{1/\alpha} t}$, and call $J_1$ topology the usual Skorokhod's topology \cite[Chapter~3]{bill}. We will prove the following result.
\begin{thm}
\label{scaling1}
If $d \ge 5$, then for almost every environment, the law of $X\e$ under $\PPt_0$ converges, for the $J_1$ topology and as $\eps$ tends to $0$, to the law of $B \circ H^{-1}$, where $B$ is a Brownian motion, $H$ is an $\alpha$-stable subordinator, and $(B,H)$ are independent.
\end{thm}
The limit process $B \circ H^{-1}$ appearing in Theorem~\ref{scaling1} is known as the \emph{fractional kinetics process} \cite{zas}. 

The proof of Theorem~\ref{scaling1} we present below can easily be adapted to cover also the Metropolis and heat bath dynamics, where one replaces the jump rates in (\ref{jumprate}) by, respectively~:
$$
\min\left(1 , \frac{\tau_y}{\tau_x}\right), \qquad \text{and} \qquad \left(1 + \frac{\tau_x}{\tau_y}  \right)^{-1}.
$$

In the case when $\E[\tau_0]$ is finite, the random walk $(X_t)_{t \ge 0}$ is diffusive under the averaged law, and converges to Brownian motion after rescaling \cite{masi}. Assumption~(\ref{regvar}), with $\alpha < 1$, brings us in the domain where the invariance principle breaks down.

The model investigated here has been considered on various graphs by physicists, as a simplified representation of the dynamics of glassy systems (see \cite{vhobc, bckm} for reviews). The first occurences concern the dynamics on the complete graph, with $a = 1$ in \cite{domin} and with $a = 0$ in \cite{bou}. The general dynamics, with $a \in [0,1]$, was considered in \cite{rmb}. The case when the underlying graph is $\Z^d$ has been studied in the physics litterature in \cite{mb, rmb, rmb01, bb}. 

The characteristic property of glassy systems is the phenomenon of aging. It is experimentally observed the following way. The glass is prepared by a fast cooling at time $t = 0$. After a time $t_w$, some experiment is performed, and a relaxation time is measured. It turns out that this relaxation time depends on the time $t_w$ that separates the instant of preparation and the experiment, on any accessible time scale. For example, one can observe the magnetic susceptibility of certain materials in presence of a small oscillating magnetic field of period $T$. It is observed that this magnetic susceptibility depends only on the ratio $T/t_w$ \cite{uppsala}. Macroscopic properties of the material thus depend on its ``age'' $t_w$.

On the mathematical side, the model attracted interest as well (see \cite{rev} for a review). In terms of the random walk, aging can be observed via two-time correlation functions, letting both times diverge to infinity. The limit obtained should be a non-trivial function of their ratio. One can derive such results from the existence of a scaling limit (see \cite[Theorem 5.1]{rev} for a simple example, and also \cite{fin, dim1, dim2, fm08}). 

In dimension $1$ and for $a = 0$, \cite{fin} obtained convergence of the rescaled process to a singular diffusion, on a subdiffusive scale. The result was extended to general $a \in [0,1]$ in \cite{dim1}. The multidimensional case was then considered. 

A fruitful approach to this problem is to introduce a time-change $(\hX_t)_{t \ge 0}$ of the initial process, in such a way that the counting measure becomes reversible for $\hX$. More precisely, we let $\hX$ follow the trajectory of $X$, but the time spent by $\hX$ at some site $x$ is the time spent by $X$ divided by $\tau_x$. For the walk $\hX$, the jump rate from a site $x$ to a neighbour $y$ is thus $(\tau_x \tau_y)^a$, which is symmetric. Letting
$$
A(t) = \int_0^{t} \tau_{\hX_s} \ \d s,
$$
we can rewrite $X$ as 
$$
X_t = \hX_{A^{-1}(t)}.
$$
We define the rescaled processes $\hX\e(t) = \sqrt{\eps} \hX_{\eps^{-1} t}$ and $H\e(t) = \eps^{1/\alpha} A(\eps^{-1} t)$. In order to prove Theorem~\ref{scaling1}, it is sufficient to prove the following result (see \cite[(3.3.4)]{whitt} for a definition of the $M_1$ topology).
\begin{thm}
\label{scaling2}
If $d \ge 5$, then for almost every environment, the joint law of $(\hX\e, H\e)$ under $\PPt_0$ converges, for the $J_1 \times M_1$ topology and as $\eps$ tends to $0$, to the law of $(B,H)$, where $B$ is a Brownian motion, $H$ is an $\alpha$-stable subordinator, and $(B,H)$ are independent.
\end{thm}
Theorem~\ref{scaling2} was obtained in~\cite{dim2} in the case when $a = 0$, in any dimension $d \ge 2$ (with a different renormalisation when $d = 2$). For $a = 0$, the time-changed random walk $\hX$ is the simple random walk. The proof is based on a coarse-graining procedure introduced in \cite{bcm}, and relies on sharp heat kernel estimates for the simple random walk. 

Recently, \cite{barcer} managed to extend the method to cover general $a \in [0,1]$, for any $d \ge 3$. A preliminary step was to obtain sharp heat kernel estimates for the time-changed random walk $\hX$ \cite{bardeu}.

Hence, Theorem~\ref{scaling1} is not new. The interest of the present paper is that we will follow a radically different method of proof, that we believe to be more natural. The main tool needed here is that the time-changed random walk $\hX$ is ``transient enough'', in the sense that the environment viewed by $\hX$ is sufficiently mixing (see Theorem~\ref{gtrajtau}). Moreover, we will see that we can in fact focus our attention on a priori weaker statements in which one considers the law of the processes under the \emph{annealed} measure $\P \PPt_0$ (that we will now write $\ov{\P}$, with corresponding expectation $\ov{\E}$). Finally, in our method, intermediate statements do not involve a mesoscopic scale, but directly the limit objects, which are more simple. We mainly focus on the following statement.
\begin{prop}
\label{scaling3}
If $d \ge 5$, then the law of $H\e$ under the measure $\ov{\P}$ converges, for the $M_1$ topology and as $\eps$ tends to $0$, to the law of an $\alpha$-stable subordinator.
\end{prop}
The process $H\e$ is an additive functional of the environment viewed by the particle $\hX$ (see (\ref{defenvpart}) for a definition). Under the annealed measure, this process is known to be stationary (Proposition~\ref{ergodicity}), and when $d \ge 5$, \cite{vardecay} provides an estimate of its speed of convergence to equilibrium (which is recalled in Theorem~\ref{gtrajtau} of the present paper). A simple consequence of these observations is the following (Proposition~\ref{indepincr}).
\begin{prop}
\label{indepincrintro}
Assume $d \ge 5$. If the law of $(H\e)$ under $\ov{\P}$ converges, along some subsequence, to the law of $H$, then $H$ is a subordinator. 
\end{prop}

Let us sketch, very roughly as a first step, our proof of Proposition~\ref{scaling3}. Tightness of the family $(H\e)_{\eps > 0}$ can easily be obtained. Indeed, the number of sites visited grows linearly with time, and the time spent on one site is bounded, so the process $H\e$ can be compared with a (rescaled) sum of i.i.d.~random variables with tail described by~(\ref{regvar}). 

The main goal is thus to show that the limit subordinator appearing in Proposition~\ref{indepincrintro} is an $\alpha$-stable subordinator, whose law does not depend on the subsequence considered. In the integral defining $H\e(t)$, only the few deepest traps visited really matter, as is usual when considering sums of non-integrable random variables. We may as well consider $H\e(t)$ as the sum of a finite number of contributions from the few traps whose depths are of the order of $\eps^{-1/\alpha}$. From Proposition~\ref{indepincrintro}, we learn that these contributions are asymptotically independent and identically distributed. It is thus sufficient to show that the contribution of the first deep trap visited converges, as $\eps$ tends to $0$. This contribution can be decomposed as a product of two terms~: first, the depth of the deep trap encountered, and second, the time spent on it by the random walk $\hX$. One can see that these two quantities become asymptotically independent, so what finally remains is to have some information on the time spent by the random walk $\hX$ on a deep trap. This occupation time is described by the Green function at this point, which can be seen as a function of the environment around the deep trap. Let us write $\tau\e(1)$ for the environment around the first deep trap encountered, and $G(\tau)$ for the Green function in the environment $\tau$. In order to describe the law of $\tau\e(1)$, we let $h : \Omega \to \R$ be some arbitrary function such that $h(\tau)$ does not depend on $\tau_0$. We study the integral of the image by $h$ of the environment viewed by the particle, but retaining only the contribution of the deepest traps. We can compute the expectation of this additive functional in two different ways. On one hand, this expectation should increase with time proportionally to $\E[h(\tau)]$, due to the stationarity of the environment. On the other, after the visit of the first deep trap, it should also be equal to $\ov{\E}[(Gh)(\tau\e(1))]$. The equality
$$
C \E[h(\tau)] = \ov{\E}[(Gh)(\tau\e(1))],
$$
which is valid for any function $h$, fully determines the law of the enviroment around the deep trap, and thus ends the proof of Proposition~\ref{scaling3}.

We now describe our proof with more precision. We hope to convince the reader that, although a bit long, it is fairly elementary. From now on, we assume that $a \neq 0$ (we will indicate briefly later on how to modify the proof to cover the special case $a = 0$). We say that a site $x \in \Z^d$ is \emph{discovered before time} $t$ if it belongs to the $1$-neighbourhood of the trajectory up to time $t$~:
$$
\exists s \le t : \|\hX_s - x\| \le 1 .
$$
We define $(x_i)_{i \in \N}$ as the sequence of sites discovered by the random walk $\hX$, following the order of appearance with time, and without repetition. The central property of interest to us is that, under $\ov{\P}$, the random variables $(\tau_{x_i})_{i \in \N}$ are independent and identically distributed (Proposition~\ref{tauiid}). Letting $l(x,\cdot)$ be the local time at site $x$ of the random walk $\hX$, one can rewrite $H\e$ as
\begin{equation}
\label{Hesumi}
H\e(t) = \eps^{1/\alpha} \int_0^{\eps^{-1} t}  \tau_{\hX_s} \ \d s = \eps^{1/\alpha} \sum_{i = 1}^{+\infty} l(x_i, \eps^{-1} t) \ \tau_{x_i}.
\end{equation}
As the number of sites discovered by $\hX$ grows linearly with time, the sum has of order $\eps^{-1}$ non-zero terms. Because of assumption~(\ref{regvar}), the main contribution in this sum comes from the largest terms. Fixing some small parameter $\delta > 0$, we say that a site $x$ is a \emph{deep trap} if $\eps^{1/\alpha} \tau_x \ge \delta$. The process $H\e$ in (\ref{Hesumi}) is well approximated by 
\begin{equation}
\label{defHde}
H\de(t) = \eps^{1/\alpha} \int_0^{\eps^{-1} t}  \tau_{\hX_s} \1_{\{ \eps^{1/\alpha} \tau_{\hX_s} \ge \delta \}} \ \d s =  \eps^{1/\alpha} \sum_{x_i \text{ deep trap}} l(x_i, \eps^{-1} t) \ \tau_{x_i}.
\end{equation}
Letting $(x\de(n))_{n \in \N}$ be the sequence of deep traps discovered by the random walk, one can rewrite (\ref{defHde}) as
\begin{equation}
\label{Hdesum}
H\de(t) = \eps^{1/\alpha} \sum_{n = 1}^{+\infty} l(x\de(n), \eps^{-1} t) \ \tau_{x\de(n)}.
\end{equation}
When a deep trap is discovered, the random walk $\hX$ may visit it several times (or possibly never), and then never return to it. These successive nearby visits occur on a time scale that does not depend on $\eps$. Due to the time renormalisation, the function $l(x\de(n), \eps^{-1} \cdot)$ tends to look more and more like a step function. Letting $T\de(n)$ be the instant of discovery of $x\de(n)$, we have (Proposition~\ref{distM1})~:
\begin{equation}
\label{ldestep}
l(x\de(n), \eps^{-1} t) \simeq l(x\de(n), \infty) \ \1_{\{t \ge \eps T\de(n)\}}.
\end{equation}
Let $G(\tau)$ be the expected total time spent at the origin for the walk $\hX$ started at the origin (in the environment $\tau$). Introducing a new random variable $e\de(n)$, we decompose $l(x\de(n), \infty)$ the following way~:
\begin{equation*}
l(x\de(n), \infty) = G\left( \theta_{x\de(n)} \ \tau \right) \ e\de(n),
\end{equation*}
where $(\theta_x)_{x \in \Z^d}$ are the translations acting on the set of environments, $(\theta_x \ \tau)_y = \tau_{x+y}$. It might happen that $l(x\de(n), \infty) = 0$, in the case when the $n^\text{th}$ deep trap is discovered but not visited by the random walk. However, because $a \neq 0$, the deep trap is more attractive than other sites, so it will actually be visited with high probability, and on this event, $e\de(n)$ is an exponential random variable of parameter $1$ (Proposition~\ref{visittrap}).

We should investigate the behaviour of $G\left( \theta_{x\de(n)} \ \tau \right)$. Let us define
$$
\tau\de(n) = \left( \tau_{x\de(n) + z} \right)_{z \neq 0} \qquad \text{and} \qquad \ov{G}\left((\tau_z)_{z \neq 0}\right) = \lim_{\tau_0 \to +\infty} G(\tau).
$$
In words, $\tau\de(n)$ is the environment around the deep trap $x\de(n)$, without consideration of the value at the origin. It turns out (Proposition~\ref{lde}) that $G\left( \theta_{x\de(n)} \ \tau \right)$ is well approximated by $\ov{G}\left(\tau\de(n)\right)$. Hence, the process $H\de(t)$ is close to
\begin{equation}
\label{Hdesuma}
\eps^{1/\alpha} \sum_{n = 1}^{+\infty}  \ov{G}\left(\tau\de(n)\right) \ e\de(n) \ \tau_{x\de(n)} \ \1_{\{t \ge \eps T\de(n)\}}.
\end{equation}

We can describe precisely the asymptotic behaviour of the random variables appearing in the above sum. Indeed, As was briefly mentioned before, the law of $e\de(n)$ converges to an exponential law of parameter $1$. Moreover, the law of the normalised depth of a deep trap $\eps^{1/\alpha} \tau_{x\de(n)}$ is the law of $\tau_0$ conditioned on being greater than $\delta \eps^{-1/\alpha}$, and hence (Proposition~\ref{taulim}) it converges to the law with density
\begin{equation}
\label{densitelimite}
 \frac{\alpha \delta^\alpha}{x^{\alpha+1}} \ \d x \ \1_{[\delta, +\infty)}(x).
\end{equation}
Let us consider the random variables $(\eps T\de(n))_{n \in \N}$. We recall that $(\tau_{x_i})_{i \in \N}$ are independent and identically distributed. Considering whether $x_i$ is a deep trap or not forms a sequence of Bernoulli trials. The probability of success is given by
$$
\P[\tau_0 \ge \eps^{-1/\alpha} \delta] \sim \frac{\eps}{\delta^\alpha}.
$$
Moreover, the number of sites visited during the time $\eps^{-1}t$ is asymptotically of order $c \eps^{-1} t$ (Proposition~\ref{range}). It is then a classical fact that $(\eps T\de(n))_{n \in \N}$ converge to a Poisson process of intensity $c\delta^{-\alpha}$ (Proposition~\ref{poisson}). 

A consequence of these simple observations, together with the fact (given by Proposition~\ref{tightenv}) that $(\tau\de(n))_{\eps > 0}$ is tight, ensures that the family of laws of
\begin{equation}
\label{ajoint}
\left( \tau\de(n), e\de(n), \eps^{1/\alpha} \tau_{x\de(n)}, T\de(n) \right)_{n \in \N},
\end{equation}
indexed by $\eps$, is tight. Let us choose a sequence $(\eps_k)_{k \in \N}$ going to zero, on which the joint law of the random variables in (\ref{ajoint}) converges to the law of some
\begin{equation}
\label{ajointlim}
\left(\tau_\delta(n), e_\delta(n), \tau_\delta^\circ(n), T_\delta(n) \right)_{n \in \N}. 
\end{equation}
By the approximation (\ref{Hdesuma}), the law of $H\dek$ then converges to the one of $H_\delta$ defined by (Proposition~\ref{limitHform})
$$
H_\delta(t) = \sum_{n = 1}^{+\infty}  \ov{G}\left(\tau_\delta(n)\right) \ e_\delta(n) \ \tau^\circ_\delta(n) \ \1_{\{t \ge  T_\delta(n)\}}.
$$
As explained before for $H$ in Proposition~\ref{indepincrintro}, we know a priori that $H_\delta$ is a subordinator. Let $\psi_\delta$ be its Laplace exponent, so that for any $\lambda > 0$~:
$$
\ov{\E}[e^{-\lambda H_\delta(t)}] = e^{-t \psi_\delta(\lambda)}.
$$
The jump rate of $H_\delta$ is the intensity of the Poisson process $(T_\delta(n))$, which is equal to $c \delta^{-\alpha}$. Hence, its Laplace transform is given by
\begin{equation}
\label{psidellam}
\psi_\delta(\lambda) = c \delta^{-\alpha} \ov{\E}\left[ 1 - e^{-\lambda \ov{G}\left(\tau_\delta(1)\right) \ e_\delta(1) \ \tau^\circ_\delta(1)}  \right].
\end{equation}
Let us now temporarily admit the following result.
\begin{prop}
\label{lawtaudelta}
The law of $\tau_\delta(1)$ is absolutely continuous with respect to $\P$, and its density is given by
$$
\frac{1}{c \ov{G}(\tau)} \d \P(\tau).
$$
\end{prop}
We have seen previously that $e_\delta(1)$ is an exponential random variable of parameter~$1$, and that the law of $\tau^\circ_\delta(1)$ is given by~(\ref{densitelimite}). Moreover, one can show (Proposition~\ref{trioindep}) that $\tau_\delta(1)$, $e_\delta(1)$ and $\tau^\circ_\delta(1)$ are independent random variables. As a consequence, the law of $H_\delta$ is fully characterized, and in particular, does not depend on the sequence $(\eps_k)$ considered. This implies that the law of $H\de$ converges to the law of $H_\delta$ (Proposition~\ref{cvHde}), and its Laplace transform (\ref{psidellam}) can be computed. Letting $\delta$ tend to $0$, one obtains (Proposition~\ref{cvH}) that $\psi_\delta(\lambda)$ converges to
\begin{equation}
\label{psilim}
\psi(\lambda) = \Gamma(\alpha+1) \E\left[\ov{G}(\tau)^{\alpha - 1}\right] \int_0^{+\infty} (1-e^{-\lambda u}) \frac{\alpha}{u^{\alpha + 1}} \ \d u,
\end{equation}
where $\Gamma$ is Euler's Gamma function. By an interversion of limits (Proposition~\ref{p:diagram}), one can check that the law of $H\e$ converges, as $\eps$ tend to $0$, to the law of an $\alpha$-stable subordinator whose Laplace exponent $\psi$ is given by~(\ref{psilim}).

Proposition~\ref{scaling3} is thus obtained, provided we can prove Proposition~\ref{lawtaudelta}. In order to do so, our approach is fairly similar. Let us write $(\htau(t))_{t \ge 0}$ for the environment viewed by $\hX$ (defined in~(\ref{defenvpart})). We say that a function $h : \Omega \to \R$ is a \emph{test function} if it is a bounded continuous function taking values in $(0,+\infty)$, and such that $h(\tau)$ does not depend on $\tau_0$. We also introduce
\begin{equation}
\label{defLde}
L_\delta^{(\eps)}(t) = \int_0^{\eps^{-1} t} h(\htau(s)) \1_{\{ \eps^{1/\alpha} \tau_{\hX_s} \ge \delta \}} \ \d s.
\end{equation}
We may decompose $L\de$ the same way we decomposed $H\de$, namely
$$
L\de(t) \simeq \sum_{n = 1}^{+\infty} (\ov{G} h)(\tau\de(n)) \ e\de(n)  \ \1_{\{t \ge \eps T\de(n)\}}.
$$
For a sequence $(\eps_k)$ along which the joint law of the random variables in (\ref{ajoint}) converges to the law of those appearing in (\ref{ajointlim}), the process $L\dek$ converges to
\begin{equation}
\label{limitLdelta}
L_\delta(t) = \sum_{n = 1}^{+\infty} (\ov{G} h)(\tau_\delta(n)) \ e_\delta(n)  \ \1_{\{t \ge T_\delta(n)\}}.
\end{equation}
The nice property of the process $L\de$ is that its expectation is easily computed. Indeed, from (\ref{defLde}) and using the stationarity of the environment viewed by the particle, it comes that
$$
\ov{\E}\left[ L\de(t) \right] = \eps^{-1} t \E\left[ h(\tau) \1_{\{ \eps^{1/\alpha} \tau_0 \ge \delta \}} \right],
$$
which, as $h(\tau)$ does not depend on $\tau_0$, and using~(\ref{regvar}), converges to
\begin{equation}
\label{expectL1}
t \delta^{-\alpha} \E[h(\tau)].
\end{equation}
On the other hand, an adaptation of Proposition~\ref{indepincrintro} shows that $L_\delta$ is a subordinator. We thus obtain from~(\ref{limitLdelta}) that
$$
\ov{\E}\left[L_\delta(t)\right] = t c \delta^{-\alpha} \ov{\E}\left[e_\delta(1) \ (\ov{G} h)(\tau_\delta(1))\right] = t c \delta^{-\alpha} \ov{\E}\left[(\ov{G} h)(\tau_\delta(1))\right],
$$
using the fact that $e_\delta(1)$ is independent of $\tau_\delta(1)$. Comparing this with (\ref{expectL1}) gives Proposition~\ref{lawtaudelta}. 

Let us now explain how one can deduce Theorem~\ref{scaling2} from Proposition~\ref{scaling3}. The first step is to obtain the convergence of the joint process $Z\e = (\hX\e, H\e)$ under the annealed measure. We take for granted that the law of $\hX\e$ converges to the law of a Brownian motion \cite{bardeu}. As a consequence, the process $Z\e$ is tight, and any limit point is of the form $Z = (B,H)$, where the laws of the marginals $B$ and $H$ are known. Moreover, a simple modification of Proposition~\ref{indepincrintro} shows that the process $Z$ has independent increments. It is a classical fact that, if a Brownian motion and a Poisson process share a common filtration, then the processes are independent. Similarly, the processes $B$ and $H$ must in fact be independent (Proposition~\ref{cvjointannealed}). This property gives a complete description of the limit law $Z$, which is therefore unique. 

The second (and last) step in order to obtain Theorem~\ref{scaling2} is to transform convergence under the annealed measure into convergence under the \emph{quenched} measure, i.e. under $\PPt_0$ for almost every $\tau$. This can be obtained by a kind of concentration argument that is due to \cite{bs}. It consists in checking that the variance of certain functionals of $Z\e$ decays sufficiently fast as $\eps$ tends to $0$, a fact ensured by Theorem~\ref{gtrajtau}.

In this paper, we focus on the case $a \neq 0$. Our proof of Proposition~\ref{scaling3} can however be easily adapted to cover the case $a = 0$, changing the sequence of discovered sites $(x_i)_{i \in \N}$ for the sequence of distinct sites actually visited by the walk. The end of the proof is then easier, as $G(\tau)$ does not depend on $\tau$ (in this case, we recall that $\hX$ is the simple random walk). Our method is however not the wisest in this case, and one can in fact obtain much better results for general random walks whose law does not depend on the environment \cite{fmv}.

An interesting feature of our results is that the limit $\alpha$-stable process $H$ appearing in Theorem~\ref{scaling2} is described explicitely by its Laplace transform~(\ref{psilim}). One can check that it coincides with the one obtained when $a = 0$ in \cite{dim2}, provided one adds the missing $G_d(0)^{-1}$ in \cite[(4.15)]{dim2}, and propagates changes accordingly.

Let us now say a word on the topologies we consider. Results of convergence concerning  processes defined on $\R_+$ should be understood as the convergence of the restrictions on $[0,t]$, for any $t > 0$. In the results mentioned above, there appears the usual Skorokhod's $J_1$ topology, and also the weaker $M_1$ topology. In Proposition~\ref{scaling3} and Theorem~\ref{scaling2}, it is not possible to replace the $M_1$ topology by the $J_1$ topology \cite{dim2} (see also the discussion at the beginning of section~\ref{s:jumps}). One may also want to replace the $J_1$ topology involved in Theorems~\ref{scaling1} and~\ref{scaling2} by the uniform topology, but measurability problems preclude this possibility \cite[Section~18]{bill}. The change of topology can nevertherless be done if one replaces the discontinuous processes $X\e$ and $\hX\e$ by continuous approximations of them.

Apart from this introduction, the paper is divided into 10 sections and an appendix. In section~\ref{s:envpart}, we recall the definition of the process of the environment viewed by the particle, and state its main properties~: stationarity, ergodicity, and mixing. In section~\ref{s:indep}, we prove a general form of Proposition~\ref{indepincrintro}. We then define the exploration process $(x_i)_{i \in \N}$ in section~\ref{s:explo}, and show that the number of sites discovered grows asymptotically linearly with time in section~\ref{s:range}. Section~\ref{s:jumps} justifies the heuristic observation~(\ref{ldestep}). The main achievement of section~\ref{s:envatrap} is to prove Proposition~\ref{lawtaudelta}, which enables us to prove Proposition~\ref{scaling3} in section~\ref{s:identif}. The convergence of the joint process $Z\e$ under the annealed measure is then derived in section~\ref{s:joint}, and the passage from the annealed to the quenched measure leading to Theorem~\ref{scaling2} is proved in section~\ref{s:quenched}. Finally, we prove that Theorem~\ref{scaling2} implies Theorem~\ref{scaling1} in section~\ref{s:concl}. Section~\ref{s:trans} is an appendix containing some classical results of interest from potential theory.

As the reader has probably already noticed, we adopt here the (unusual for the author) convention that $\N = \{1,2,\ldots\}$.  We avoid this way slightly awkward formulations concerning for instance the ``$0^\text{th}$'' deep trap discovered.
%
%
%
%
%
%
\section{The environment viewed by the particle}
\label{s:envpart}
\setcounter{equation}{0}
We recall here the definition of the environment viewed by the particle, as well as some of its important properties. There are translations $(\theta_x)_{x \in \Z^d}$ acting on the space $\Omega$ of environments, such that $(\theta_x \ \tau)_y = \tau_{x+y}$. The \emph{environment viewed by the particle} is the Markov process on $\Omega$ defined by
\begin{equation}
\label{defenvpart}
\htau(t) = \theta_{\hX_t} \ \tau.
\end{equation}
We recall the following classical result \cite[Lemma 4.3 (iv)]{masi}.
\begin{prop}
\label{ergodicity}
The measure $\P$ is reversible and ergodic for the process $(\htau(t))_{t \ge 0}$.
\end{prop}
Throughout this paper, one central tool is an estimate of the speed of convergence to equilibrium of this process \cite[Proposition 7.2]{vardecay}, that we now recall. For some $s \ge 0$, we say that a function $g(\hX,\tau)$ \emph{depends only on the trajectory up to time} $s$ if one can write it as
\begin{equation}
\label{gform}
g\left((\hX_u)_{u \le s},(\tau_{\hX_u})_{u \le s}\right).
\end{equation}
We say that such a function is \emph{translation invariant} if moreover, for any $x \in \Z^d$~:
$$
g\left((x + \hX_u)_{u \le s},(\tau_{\hX_u})_{u \le s}\right) = g\left((\hX_u)_{u \le s},(\tau_{\hX_u})_{u \le s}\right).
$$
For a function $f : \Omega \to \R$, we write $\var(f)$ for the variance of the function $f$ with respect to the measure $\P$, and $f_t = \EEt_0[f(\htau(t))]$. \cite[Proposition 7.2]{vardecay} states the following.
\begin{thm}
\label{gtrajtau}
When $d \ge 5$, there exists $C > 0$ such that, for any bounded function $g$ that depends only on the trajectory up to time $s$ and is translation invariant, if $f(\tau) = \EEt_0[g]$, then for any $t > 0$~:
$$
\var(f_t) \le C \|g \|_\infty^2  \ \frac{(s+t)^2}{t^{d/2}} .
$$
\end{thm}
We point out that this result was initially established for a random walk among random independent conductances. Although $\hX$ is a random walk among random conductances, they fail to be independent. Indeed, the conductance of the edge between two sites $x \sim y$ is $(\tau_x \tau_y)^a$, hence conductances of adjacent edges are correlated. One can however check that, due to the locality of the dependence dealt with here, the results of \cite{vardecay} still apply in our present context. 

For convenience, we also recall a useful observation from \cite[(7.2)]{vardecay}. For a function $g$ of the form (\ref{gform}), we write $g(t)$ for
$$
g\left((\hX_u)_{t \le u \le t + s},(\tau_{\hX_u})_{t \le u \le t + s}\right),
$$
and define $f(\tau) = \EEt_0[g]$. Then one can check that, when $g$ is translation invariant~:
\begin{equation}
\label{vareq72}
f_t(\tau) = \EEt_0[g(t)].
\end{equation}
%
%
%
%
%
%
\section{Asymptotically independent increments}
\label{s:indep}
\setcounter{equation}{0}
Let $\delta \ge 0$. We recall from~(\ref{defHde}) the definition of $H\de$ as
$$
H\de(t) = \eps^{1/\alpha} \int_0^{\eps^{-1} t}  \tau_{\hX_s} \1_{\{ \eps^{1/\alpha} \tau_{\hX_s} \ge \delta \}} \ \d s.
$$
In this section, we consider possible limit points for the law of $H\de$ under $\ov{\P}$, as $\eps$ goes to $0$. We will see, using Theorem~\ref{gtrajtau}, that any such limit point is the law of a subordinator. More precisely, for a sequence $(\eps_k)_{k \in \N}$ converging to $0$, let us assume that the law of $\Hdek$ under $\ov{\P}$ converges, in the sense of finite-dimensional distributions and as $k$ tends to infinity, to some measure $\mu_\delta^\circ$. Possibly enlarging the probability space, we define a random variable $\Hdo$ which has law $\mu_\delta^\circ$ under $\ov{\P}$. The purpose of this section is to show the following result.
\begin{prop}
\label{indepincr}
If $d \ge 5$, then the process $\Hdo$ is a subordinator under $\ov{\P}$.
\end{prop}
\begin{proof}
Under $\ov{\P}$, the process $\Hde$ inherits the stationarity property from that of $(\htau(t))_{t \ge 0}$. 
What we need to see is the independence of the increments in the limit. We will prove, by induction on $n$, that for any $\lambda_1,\ldots,\lambda_n$, and any $s_1 < \cdots < s_{2n}$~:
\begin{equation}
\begin{split}
& \ov{\E}\left[\exp\left(-\sum_{i=1}^n \lambda_i (\Hdo(s_{2i}) - \Hdo(s_{2i-1}))\right)\right] \\
& \qquad  = \prod_{i = 1}^n \ov{\E}\left[\exp\big(\lambda_i (\Hdo(s_{2i}) - \Hdo(s_{2i-1})\big)\right].
\end{split}
\end{equation}
The property is obvious when $n = 1$. We assume it to hold up to $n$, and give ourselves $\lambda_1,\ldots,\lambda_{n+1}$, and $s_1 < \cdots < s_{2n+2}$. 
We define 
$$
P_k = \exp\left(-\sum_{i=1}^n \lambda_i (\Hdek(s_{2i}) - \Hdek(s_{2i-1}))\right),
$$
$$
F_k(\tau) = \EEt_0\left[\exp\big(-\lambda_{n+1} (\Hdek(s_{2n+2}-s_{2n}) - \Hdek(s_{2n+1}-s_{2n})\big)\right].
$$
Using the Markov property at time $\eps^{-1} s_{2n}$, we have that
$$
\EEt_0\left[\exp\left(-\sum_{i=1}^{n+1} \lambda_i (\Hdek(s_{2i}) - \Hdek(s_{2i-1}))\right)\right] = \EEt_0[P_k F_k(\htau(\eps^{-1} s_{2n}))],
$$
where we recall that $\htau(\eps^{-1} s_{2n})$ is the environment seen by the particle at time $\eps^{-1} s_{2n}$. What we want to show is precisely that
\begin{equation}
\label{diffvar}
\Big| \ov{\E}[P_k F_k(\htau(\eps^{-1} s_{2n}))] - \ov{\E}[P_k]\ov{\E}[F_k(\htau(\eps^{-1} s_{2n}))] \Big|
\end{equation}
goes to $0$ as $k$ goes to infinity. Indeed, on one hand, the induction hypothesis can be applied on the limit of $\ov{\E}[P_k]$, and on the other, one has
$$
\ov{\E}[F_k(\htau(\eps^{-1} s_{2n}))] = \ov{\E}\left[\exp\left(- \lambda_{n+1} (\Hdek(s_{2n+2}) - \Hdek(s_{2n+1}))\right)\right].
$$
But, by Cauchy-Schwarz inequality, the square of the term in (\ref{diffvar}) is bounded by the product of the variances of $P_k$ and $F_k(\htau(\eps^{-1} s_{2n}))$. As $P_k$ is bounded, it is enough to show that the variance of $F_k(\htau(\eps^{-1} s_{2n}))$ goes to $0$ as $k$ tends to infinity. Due to the stationarity of $(\htau(t))$, the variance of $F_k(\htau(\eps^{-1} s_{2n}))$ is the same as the variance of $F_k(\tau)$. 

Hence, we now proceed to show that $\var(F_k)$ goes to $0$ as $k$ tends to infinity. We begin by rewriting $F_k$, using the definition of $H\de$ given in~(\ref{defHde}), as~:
$$
F_k(\tau) = \EEt_0\left[\exp\left(-\lambda_{n+1} \int_{\eps_k^{-1}(s_{2n+1} - s_{2n})}^{\eps^{-1} (s_{2n+2}-s_{2n})} \eps_k^{1/\alpha} \tau_{\hX_s} \1_{\{ \eps_k^{1/\alpha} \tau_{\hX_s} \ge \delta \}} \ \d s \right)\right].
$$
We define $g_k$ as
$$
g_k(\hX,\tau) = \exp\left(-\lambda_{n+1} \int_{0}^{\eps_k^{-1} (s_{2n+2}-s_{2n+1})} \eps_k^{1/\alpha} \tau_{\hX_s} \1_{\{ \eps_k^{1/\alpha} \tau_{\hX_s} \ge \delta \}} \ \d s \right).
$$
The function $g_k$ depends only on the trajectory up to time $\eps_k^{-1}(s_{2n+2}-s_{2n+1})$, and it is translation invariant. Therefore, letting $f^{(k)}(\tau) = \EEt_0[g_k]$, and $f^{(k)}_t(\tau) = \EEt_0[f^{(k)}(\htau(t))]$, we have (see (\ref{vareq72}))~:
$$
F_k(\tau) = f^{(k)}_{\eps_k^{-1}(s_{2n+1}-s_{2n})}(\tau),
$$
and Theorem~\ref{gtrajtau} implies that, when $d \ge 5$,
$$
\var(F_k) \le C \eps_k^{d/2-2} \frac{(s_{2n+2} - s_{2n})^2}{(s_{2n+1} - s_{2n})^{d/2}},
$$
which goes indeed to $0$ as $k$ tends to infinity.
\end{proof}

The same technique applies as well for $\Lde$ defined in (\ref{defLde}).
\begin{prop}
\label{indepincrL}
If $d \ge 5$, then any limit point (in the sense of the convergence of finite-dimensional distributions) of the laws of $\Lde$ under $\ov{\P}$ is the law of a subordinator.
\end{prop}

%
%
%
%
%
%
\section{The exploration process}
\label{s:explo}
\setcounter{equation}{0}

In this section, we define a way to explore the $1$-neighbourhood of the trajectory of the walk, and state some of its properties. Let $\gamma = (\gamma_n)$ be a (finite or infinite) nearest-neighbour path on $\Z^d$. The set of sites we would like to explore is
\begin{equation}
\label{defmcD}
\mcD(\gamma) = \{ x \in \Z^d : \exists n \quad \|x - \gamma_n \| \le 1 \}.
\end{equation}
Consider the sequence of sites~:
$$
\mathcal{S} = (\gamma_1+z)_{|z |\le 1}, (\gamma_2+z)_{|z| \le 1}, \ldots, (\gamma_n+z)_{|z| \le 1}, \ldots
$$
where $(x+z)_{|z| \le 1}$ is enumerated in some predetermined order. It is clear that $\mathcal{S}$ spans $\mcD(\gamma)$. We can define 
\begin{equation}
\label{defxn}
\mathcal{S'} = x_1(\gamma), \ldots, x_{n}(\gamma), \ldots
\end{equation}
as the sequence $\mathcal{S}$ with repetitions removed. We call $(x_n(\gamma))$ the \emph{exploration process} (for the path $\gamma$), and say that a site is \emph{discovered} (by the path $\gamma$) if it belongs to $\mcD(\gamma)$. 

Let $Y$ be the discrete-time random walk associated with $\hX$. We will be mainly interested in the exploration process associated with the random walk, namely $(x_n(Y))_{n \in \N}$. As $Y$ is an irreducible Markov chain on an infinite state space, it is clear that it visits an infinite number of distinct sites, which implies that $x_n(Y)$ is well defined for all $n$.

We write $\mu_0$ for the law of $\tau_0$ under the measure $\P$. 
\begin{prop}
\label{tauiid}
Under $\ov{\P}$, the random variables $(\tau_{x_n(Y)})_{n \in \N}$ are independent and distributed according to $\mu_0$.
\end{prop}
\begin{proof}
Let $\gamma = (\gamma_i)_{1 \le i \le k}$ be a nearest-neighbour path, and $\gamma^{\leftarrow} = (\gamma_i)_{1 \le i \le k-1}$ be the (possibly empty) path that follows $\gamma$ but stops one step earlier. Let $E_n$ be the set of paths such that $\gamma$ discovers at least $n$ sites $x_1(\gamma),\ldots,x_n(\gamma)$, but $\gamma^{\leftarrow}$ does not. 

For $\gamma = (\gamma_i)_{1 \le i \le k}$, we write $Y = \gamma$ when $Y$ and $\gamma$ coincide up to time $k$. Observe that the events $\{Y = \gamma\}$ form a partition of the probability space when $\gamma$ ranges in $E_n$ (up to a set of null measure). For any real numbers $t_1,\ldots, t_n$, we can therefore decompose 
\begin{equation}
\label{indeptau}
\ov{\P}[\tau_{x_1(Y)} \le t_1,\ldots, \tau_{x_n(Y)} \le t_n]
\end{equation}
along this partition, which leads to~:
\begin{equation*}
\begin{split}
& \sum_{\gamma \in E_{n}} \P\PPt_0[Y=\gamma, \tau_{x_1(\gamma)} \le t_1,\ldots, \tau_{x_n(\gamma)} \le t_n] \\
& \qquad = \sum_{\gamma \in E_{n}} \E\left[ \PPt_0[Y=\gamma, \tau_{x_1(\gamma)} \le t_1,\ldots, \tau_{x_{n-1}(\gamma)} \le t_{n-1}] \ \1_{\{\tau_{x_n(\gamma)} \le t_n\}} \right] ,
\end{split}
\end{equation*}
as $\tau_{x_n(\gamma)}$ is constant under $\PPt_0$. The event $\{Y = \gamma \}$ depends only on $(\tau_x)_{x \in \mathcal{D}(\gamma^\leftarrow)}$. As $\gamma \in E_{n}$, the shortened path $\gamma^\leftarrow$ does not discover $x_n(\gamma)$. Hence, of the two terms in the $\E$ expectation above, the first depends only on $(\tau_{x_k(\gamma)})_{k < n}$, while the second depends only on $\tau_{x_n(\gamma)}$. These two terms are thus independent, and one can rewrite the whole sum as
$$
\sum_{\gamma \in E_{n}} \P\PPt_0[Y=\gamma, \tau_{x_1(\gamma)} \le t_1,\ldots, \tau_{x_{n-1}(\gamma)} \le t_{n-1}] \ \P[\tau_{x_n(\gamma)} \le t_n].
$$
Using the translation invariance of $\P$, the probability (\ref{indeptau}) thus equals
$$
\ov{\P}[\tau_{x_1(Y)} \le t_1,\ldots, \tau_{x_{n-1}(Y)} \le t_{n-1}] \ \P[\tau_0 \le t_n].
$$
The proof of the proposition is then obtained by induction.
\end{proof}

From now on, we will simply write $x_n$ for $x_n(Y)$.
%
%
%
%
%
%
\section{Asymptotic behaviour of the range}
\label{s:range}
\setcounter{equation}{0}

We say that $x \in \Z^d$ is \emph{discovered before time }$t$ if $x \in \mcD((\hX_s)_{s \le t})$. Let $r(t)$ be the number of such sites~:
\begin{equation}
\label{defrt}
r(t) = | \mcD((\hX_s)_{s \le t}) | .
\end{equation}

In this section, we will show a law of large numbers for $r(t)$.
In order to do so, a convenient tool is the subadditive ergodic theorem. Its use requires that we ensure first that $r(t)$ is integrable. We will in fact find an upper bound for $\E[r(t)^2]$, which will be useful later on in order to show the uniform integrability of $(L\de(t))_{\eps > 0}$ (Lemma~\ref{unifint}).
\begin{prop}
\label{discovered}
There exists a constant $C > 0$ such that, for any $t \ge 0$~:
$$
\ov{\E}[r(t)^2] \le C (t+1)^2.
$$
\end{prop}
\begin{proof}
The proof is similar to the one of the second part of \cite[Proposition 6.2]{vardecay}. Let $R(t)$ be the cardinal of the range of the random walk at time $t$ (without taking into account sites that are discovered but not visited)~:
$$
R(t) = \left| \{ \hX_s, s \le t\} \right|.
$$
It is clear from the definition (\ref{defmcD}) that there are at most $(2d+1)$ sites discovered associated to each site visited, so that
$$
r(t) \le (2d+1) R(t).
$$
We will therefore focus on $R(t)$. For $x \in \Z^d$, let $\mathcal{T}_x$ be the hitting time of $x$~:
\begin{equation}
\label{defmclT}
\mathcal{T}_x = \inf \{ s \ge 0 : \hX_s = x \}.
\end{equation}
One can rewrite the range as
\begin{equation}
\label{decompRt}
R(t) = \sum_{x \in \Z^d} \1_{\{ \mathcal{T}_x \le t \}},
\end{equation}
from which we derive an expression for the square of $R(t)$~:
\begin{equation}
\label{decompRt2}
R(t)^2 = 2 \sum_{x,y \in \Z^d} \1_{ \{ \mathcal{T}_x \le \mathcal{T}_y \le t \}}.
\end{equation}
The general term of this sum can be bounded from above, using the Markov property of the random walk at time $\mathcal{T}_x$, as follows~:
\begin{eqnarray*}
\PPt_0[\mathcal{T}_x \le \mathcal{T}_y \le t] & \le & \PPt_0[\mathcal{T}_x \le t, \ \mathcal{T}_x \le \mathcal{T}_y \le \mathcal{T}_x + t ] \\
& \le & \PPt_0[\mathcal{T}_x \le t] \ \PPt_x[\mathcal{T}_y \le t].
\end{eqnarray*}
Hence, using equations (\ref{decompRt2}) and then (\ref{decompRt}), we obtain that
\begin{eqnarray}
\label{decompdiscov1}
\EEt_0[R(t)^2] & \le & \sum_{x, y \in \Z^d} \PPt_0[\mathcal{T}_x \le t] \ \PPt_x[\mathcal{T}_y \le t]  \notag \\
& \le & \sum_{x \in \Z^d} \PPt_0[\mathcal{T}_x \le t] \ \EEt_x[R(t)].
\end{eqnarray}

One would like to compare the probability to hit $x$ with the total time spent on this site, which is easier to handle. However, because there exist sites with arbitrarily large jump rates, there is no clear comparison between these two quantities. We will therefore create a larger set, written $\ov{\mathcal{V}}_\tau(x)$, that contains $x$ and so that it does take some time to exit this set.

Let $p_\tau(x)$ be the total jump rate of site $x$~:
$$
p_\tau(x) = \sum_{y \sim x} (\tau_x \tau_y)^a.
$$
For $\eta > 0$, we say that a point $x \in \Z^d$ is \emph{good} in the environment $\tau$ if $p_\tau(x) \le \eta$ ; we say that it is \emph{bad} otherwise. We need to fix $\eta$ large enough, so that
\begin{equation}
\label{0isbad}
\P[0 \text{ is bad}] < \frac{1}{(2d)^{(2d+1)}}.
\end{equation}
Following \cite[(6.10)]{vardecay}, we define $\ov{\mathcal{V}}_\tau(x)$ by~:
\begin{equation}
\label{gammal}
y \in \ov{\mathcal{V}}_\tau(x) 
\Leftrightarrow \exists \gamma = (\gamma_1, \ldots, \gamma_l) : \gamma_1 = x, \gamma_l = y, \gamma_2,\ldots,\gamma_{l-1} \text{ bad points},
\end{equation}
where $\gamma$ is a nearest-neighbour path. 
One can check that $\ov{\mathcal{V}}_\tau(x)$ contains the site~$x$, and that any point in the inner boundary of $\ov{\mathcal{V}}_\tau(x)$ is a good point. As a consequence, the expected time spent inside this set cannot be too small as compared to the probability to hit $x$. More precisely, we have the following result.

\begin{lem}
\label{hitprobatime}
For almost every $\tau$, the following holds~:
\begin{equation*}
\PPt_0[\mathcal{T}_x \le t] \le e \eta \int_0^{t+1} \PPt_0[\hX_s \in \ov{\mathcal{V}}_\tau(x)] \ \d s .
\end{equation*}
\end{lem}
We refer to \cite[Lemma 6.3]{vardecay} for a proof of this lemma. Using this result back into equation (\ref{decompdiscov1}), one obtains~:
\begin{eqnarray}
\label{decompdiscov2}
\EEt_0[R(t)^2] & \le & e \eta \sum_{x \in \Z^d} \EEt_x[R(t)] \int_0^{t+1} \PPt_0[\hX_s \in \ov{\mathcal{V}}_\tau(x)]  \ \d s  \notag \\
& \le & e \eta \sum_{x,x' \in \Z^d} \EEt_x[R(t)] \1_{\{x' \in \ov{\mathcal{V}}_\tau(x)\}} \int_0^{t+1}  \PPt_0[\hX_s = x'] \ \d s .
\end{eqnarray}
We introduce the following function of the environment~:
\begin{equation}
\label{defWt}
W_t(\tau) = \sum_{x \in \Z^d} \EEt_x[R(t)] \1_{0 \in \{\ov{\mathcal{V}}_\tau(x)\}}.
\end{equation}
It is not hard to check that for any $x' \in \Z^d$, one has
$$
W_t(\theta_{x'} \ \tau) = \sum_{x \in \Z^d} \EEt_{x+x'}[R(t)] \1_{\{x' \in \ov{\mathcal{V}}_\tau(x+x')\}} = \sum_{x \in \Z^d} \EEt_x[R(t)] \1_{\{x' \in \ov{\mathcal{V}}_\tau(x)\}},
$$
and using this fact together with inequality (\ref{decompdiscov2}), one obtains~:
\begin{eqnarray*}
\ov{\E}[R(t)^2] & \le & e \eta \int_0^{t+1} \sum_{x' \in \Z^d} \ov{\E}[W_t(\theta_{x'} \ \tau) \1_{\{\hX_s = x'\}}] \ \d s \\
& \le & e \eta \int_0^{t+1} \ov{\E}[W_t(\htau(s))] \ \d s.
\end{eqnarray*}
The measure $\P$ being invariant for the process $(\htau(s))_{s \ge 0}$, we are led to
\begin{equation}
\label{justavantlemperco}
\ov{\E}[R(t)^2]  \le  e \eta (t+1) \E[W_t(\tau)].
\end{equation}
Proposition~\ref{discovered} will therefore be proved once we have shown the following lemma.
\begin{lem}
\label{lemperco}
There exists $C > 0$ such that, for any $t \ge 0$, one has
$$
\E[W_t(\tau)] \le C (t+1).
$$
\end{lem}
\begin{proof}[Proof of Lemma~\ref{lemperco}]
We will use Lemma~\ref{hitprobatime} once more. Indeed, one has
\begin{eqnarray}
\label{decompR}
\EEt_x[R(t)] & = & \sum_{y \in \Z^d} \PPt_x[\mathcal{T}_y \le t] \notag \\
& \le & e \eta \sum_{y \in \Z^d} \int_0^{t+1} \PPt_x[\hX_s \in \ov{\mathcal{V}}_\tau(y)] \ \d s \notag \\
& \le & e \eta \sum_{y,y' \in \Z^d} \1_{\{ y' \in \ov{\mathcal{V}}_\tau(y) \}} \int_0^{t+1} \PPt_x[\hX_s = y'] \ \d s.
\end{eqnarray}
Let $w(\tau)$ be defined by
$$
w(\tau) = \sum_{y \in \Z^d} \1_{\{ 0 \in \ov{\mathcal{V}}_\tau(y) \}}.
$$
Then, inequality (\ref{decompR}) can be rewritten as
$$
\EEt_x[R(t)] \le e \eta \int_0^{t+1} \EEt_x[w(\htau(s))] \ \d s.
$$
It then follows from the definition of $W_t$ in (\ref{defWt}) that
$$
\E[W_t(\tau)] \le e \eta \sum_{x \in \Z^d} \int_{0}^{t+1} \E\left[ \EEt_x[w(\htau(s))] \1_{\{0 \in \ov{\mathcal{V}}_\tau(x)\}} \right]  \ \d s.
$$
Using Cauchy-Schwarz inequality, we obtain~:
$$
\E[W_t(\tau)] \le e \eta \sum_{x \in \Z^d} \int_{0}^{t+1} \E\left[ (\EEt_x[w(\htau(s))])^2 \right]^{1/2} \P[0 \in \ov{\mathcal{V}}_\tau(x)]^{1/2}  \ \d s .
$$
Moreover, the function 
$$
s \mapsto \E\left[ (\EEt_x[w(\htau(s))])^2 \right]
$$
is decreasing, as $\EEt_x[w(\htau(s))]$ is the image of $w$ by the semi-group associated to the process $(\htau(s))_{s \ge 0}$, hence we are finally led to~:
\begin{equation}
\label{justavantdeuxieme}
\E[W_t(\tau)] \le e \eta (t+1) \E[w(\tau)^2] \sum_{x \in \Z^d} \P[0 \in \ov{\mathcal{V}}_\tau(x)]^{1/2}.
\end{equation}
What we have to prove is thus that, on one hand, the sum appearing above is finite, and on the other hand, that the random variable $w(\tau)^2$ is integrable. 

Due to the symmetry of the definition of $\ov{\mathcal{V}}_\tau(x)$ in (\ref{gammal}), it is clear that
$$
0 \in \ov{\mathcal{V}}_\tau(x) \Leftrightarrow x \in \ov{\mathcal{V}}_\tau(0),
$$
and as a consequence, $w(\tau) = |\ov{\mathcal{V}}_\tau(0)|$. Let $\mathcal{B}(r)$ be the ball of radius $r$, with respect to the graph norm. By a percolation argument and using (\ref{0isbad}) (see \cite[Lemma 6.4 (2)]{vardecay} for details), one can see that the probability that $\ov{\mathcal{V}}_\tau(0)$ is not contained in $\mathcal{B}(r)$ decays exponentially as $r$ goes to infinity. From this fact, one can check that the two conditions mentioned above hold, which ends the proof of the lemma.
\end{proof}
The result of Lemma~\ref{lemperco}, together with inequality (\ref{justavantlemperco}), implies Proposition~\ref{discovered}.
\end{proof}

Once the integrability of $r(t)$ is ensured, a law of large numbers follows as a consequence of the subadditive ergodic theorem.

\begin{prop}
\label{range}
If $d \ge 3$, then there exists a constant $c > 0$ such that
\begin{equation*}
\label{eq:range}
\frac{r(t)}{t} \ \xrightarrow[t \to +\infty]{\ov{\P}\text{-a.s.}} c.
\end{equation*}
\end{prop}
\begin{proof}
The subadditive ergodic theorem (see \cite{kingman}), together with Proposition~\ref{discovered}, ensures that there exists a random variable $c$ such that, $\ov{\P}$-almost surely and in $L^1(\ov{\P})$~:
\begin{equation}
\label{subadditive}
\frac{r(t)}{t} \ \xrightarrow[\substack{t \to +\infty \\ t \in \N}]{} c.
\end{equation}
One can in fact omit the restriction $t \in \N$ above. Indeed, if $n$ is an integer and $n \le t < n+1$, the subadditivity property gives~:
$$
0 \le r(t) - r(n) \le |\mathcal{D}((\hX_s)_{n \le s < n+1})|.
$$
Moreover, Birkhoff's ergodic theorem ensures the almost sure convergence of 
$$
\frac{1}{n} \sum_{k = 1}^n |\mathcal{D}((\hX_s)_{k \le s < k+1})|
$$
to some random variable. In particular, 
$$
\frac{1}{n} |\mathcal{D}((\hX_s)_{n \le s < n+1})|
$$
converges to $0$ almost surely, which implies that one can take out the restriction $t \in \N$ in equation (\ref{subadditive}). Moreover, a consequence of the ergodicity given by Proposition~\ref{ergodicity} is that $c$ is in fact constant. What is left is to check that $c$ is strictly positive. For any integer $n$, we have the following convenient lower bound~:
$$
r(n) \ge \sum_{k = 1}^{n} \1_{\{ \hX_{k} \notin \mcD((\hX_s)_{s \le k-1}) \}}.
$$
Indeed, the condition $\hX_{k} \notin \mcD((\hX_s)_{s \le k-1})$ implies that the site $\hX_{k}$ has been discovered in the time interval $(k-1,k]$. Integrating this inequality, we have~:
\begin{equation}
\label{ineq:c>0}
\ov{\E}[r(n)] \ge \sum_{k=1}^{n} \ov{\P}[\hX_{k} \notin \mcD((\hX_s)_{s \le k-1})].
\end{equation}
Note that, due to the reversibility of the walk,
$$
\PPt_0[\hX_{k} \notin \mcD((\hX_s)_{s \le k-1}), \hX_{k} = x] = \PPt_x[x \notin \mcD((\hX_s)_{1 \le s \le k}), \hX_{k} = 0].
$$
Using once again the fact that $(\hX_t-x)$ has same law under $\PPt_x$ as $(\hX_t)$ under $\PP^{\theta_x \tau}_0$, the latter equals~:
$$
\PP^{\theta_x \tau}_0[0 \notin \mcD((\hX_s)_{1 \le s \le k}), \hX_{k} = -x].
$$ 
Using the translation invariance of $\P$, this computation leads us to
\begin{equation*}
\begin{split}
& \ov{\P}[\hX_{k} \notin \mcD((\hX_s)_{s \le k-1}) ] \\
& \qquad = \sum_{x \in \Z^d} \P\PP^{\theta_x \tau}_0[0 \notin \mcD((\hX_s)_{1 \le s \le k}), \hX_{k} = -x] \\
& \qquad  = \sum_{x \in \Z^d} \ov{\P}[0 \notin \mcD((\hX_s)_{1 \le s \le k}), \hX_{k} = -x] \\
& \qquad = \ov{\P}[0 \notin \mcD((\hX_s)_{1 \le s \le k})] \\
& \qquad \ge \ov{\P}[0 \notin \mcD((\hX_s)_{s \ge 1})].
\end{split}
\end{equation*}
In order to show that $c$ is strictly positive, and considering inequality (\ref{ineq:c>0}), it is enough to show that $\ov{\P}[0 \notin \mcD((\hX_s)_{s \ge 1})]$ is strictly positive, which amounts to checking that the random walk is transient. This fact is contained in Proposition~\ref{compG} of the Appendix, where it is shown that the Green function $G(\tau)$ is finite.
\end{proof}

%
%
%
%
%
%

\section{Unmatched jumps}
\label{s:jumps}
\setcounter{equation}{0}
We begin by introducing some notation. 
For any $\delta > 0$ and any $\eps > 0$, we define $(r\de(n))_{n \in \N}$ as the increasing sequence that spans the set
\begin{equation}
\label{defRde}
\mathcal{R}\de := \{i \in \N : \eps^{1/\alpha} \tau_{x_i} \ge \delta \}.
\end{equation}
In other words, the $k^{\text{th}}$ site discovered by the random walk is the $n^{\text{th}}$ deep trap discovered if and only if $k = r\de(n)$. Let $T\de(n)$ be the instant when the $n^{\text{th}}$ deep trap is discovered~: 
$$
T\de(n) = \inf \{ t : r(t) \ge r\de(n) \}.
$$
We further define $x\de(n)$ to be the location of the $n^{\text{th}}$ deep trap discovered, which is equal to $x_{r\de(n)}$. Note that $x\de(n)$ and the position of the walk at the instant of discovery $\hX_{T\de(n)}$ are neighbours (or possibly equal if at the origin). The depth of the trap discovered is given by $\tau_{x\de(n)}$, and $\left(\theta_{x\de(n)} \ \tau \right)$ is the environment seen from the trap.

We would like to consider how much time is spent on a deep trap, so we introduce
$$
l\de(n,t) = \int_0^{\eps^{-1} t} \1_{\{\hX_s = x\de(n)\}} \ \d s, \qquad \text{and} \qquad l\de(n) = l\de(n,+\infty).
$$
With this notation at hand, the processes $\Hde$ and $\Lde$ introduced in (\ref{defHde}) and (\ref{defLde}) can be conveniently rewritten as
\begin{equation}
\label{Hderewrit}
\Hde(t) =  \sum_{n = 1}^{+\infty} l\de(n,t) \ \eps^{1/\alpha} \tau_{x\de(n)} ,
\end{equation}
\begin{equation}
\label{Lderewrit}
\Lde(t) = \sum_{n = 1}^{+\infty} l\de(n,t) \ h\left(\theta_{x\de(n)} \ \tau \right).
\end{equation}

Once the random walk has found a deep trap, it will perform several visits to this site, and then leave it forever. These visits happen on a time scale that does not depend on $\eps$. Hence, due to the time renormalization, the function $l\de(n,\cdot)$ tends to look more and more like a step function as $\eps$ goes to $0$. However, some caution is necessary when one wants to give a precise meaning to this closeness. Indeed, the function $l\de(n,\cdot)$ is continuous, and we recall that the set of continuous functions is closed for the usual Skorokhod's $J_1$ topology. In the terminology of \cite{whitt}, the limit process should have jumps that are unmatched in the converging processes. Following \cite{dim2}, we will use Skorokhod's $M_1$ topology, for which jumps can appear in the limit of continuous functions. From the fact that $l\de(n,\cdot)$ is close to a step function, we will be able to show that $H\de$ and $L\de$ are well approximated, respectively, by the processes $\mathcal{H}\de$ and $\mathcal{L}\de$ defined by

\begin{equation}
\label{defmclH}
\mathcal{H}\de(t) = \sum_{n = 1}^{+\infty} l\de(n) \ \eps^{1/\alpha} \tau_{x\de(n)} \ \1_{\{t \ge \eps T\de(n)\}},
\end{equation}
\begin{equation}
\label{defmclL}
\mathcal{L}\de(t) = \sum_{n = 1}^{+\infty} l\de(n) \ h\left(\theta_{x\de(n)} \ \tau \right) \ \1_{\{t \ge \eps T\de(n)\}}.
\end{equation}

This is the content of Proposition~\ref{distM1}. Before stating it, we need to show that the jump instants $(T\de(n))_{n \in \N}$ do not accumulate in the limit, and that $\eps^{1/\alpha} \tau_{x\de(n)}$ is tight. The next proposition shows in fact that, as $\eps$ gets small, the sequence of jump instants tends to behave like a Poisson process. The knowledge of the intensity of the limit process will be useful in the sequel.

\begin{prop}
\label{poisson}
Under $\ov{\P}$, the law of $(\eps T\de(n))_{n \in \N}$ converges to that of a Poisson process of intensity $c/ \delta^\alpha$.
\end{prop}
\begin{proof}
We recall from Proposition~\ref{tauiid} that $(\tau_{x_i})_{i \in \N}$ is a family of independent random variables distributed according to $\mu_0$. Hence, $(\1_{i \in \mathcal{R}\de})_{i \in \N}$ is a family of independent Bernoulli random variables of parameter 
$$
\P[\eps^{1/\alpha} \tau_{0} \ge \delta] \sim \frac{\eps}{\delta^\alpha} \qquad (\eps \to 0).
$$
It is thus clear that $r\de(1), r\de(2)-r\de(1),\ldots$ are independent and identically distributed, and that for any $y \ge 0$~:
$$
\P[r\de(1) > \eps^{-1} y] = \left(1-\P[\eps^{1/\alpha} \tau_{0} \ge \delta]\right)^{\lfloor \eps^{-1} y \rfloor} \xrightarrow[\eps \to 0]{} e^{-y/\delta^\alpha}.
$$
As a consequence, $(\eps r\de(n))_{n \in \N}$ converges in distribution to a Poisson process of intensity $\delta^{-\alpha}$. 
Moreover, note that $r(t)$ defined in (\ref{defrt}) inherits right continuity from the one of $\hX$. Besides, because the walk cannot discover more than $2d+1$ sites at once, the heights of the jumps of $r(t)$ are bounded by $2d+1$. We obtain the inequalities~:
$$
r\de(n) \le r\left( T\de(n) \right) \le r\de(n) + 2d+1.
$$
This last inequality implies that $(\eps r(T\de(n)))_{n \in \N}$ also converges in distribution to a Poisson process of intensity $\delta^{-\alpha}$. Hence, for any $n$, the random variable $T\de(n)$ goes to infinity in probability, and Proposition~\ref{range} implies the announced result.
\end{proof}

One can easily describe the limit distribution of $\eps^{1/\alpha} \tau_{x\de(n)}$.
\begin{prop}
\label{taulim}
For any $n \in \N$, the law of $\eps^{1/\alpha} \tau_{x\de(n)}$ under $\ov{\P}$ converges to the law which density is given by
$$
\alpha \delta^{\alpha} \frac{\d x}{x^{\alpha+1}} \ \1_{[\delta, +\infty)}(x).
$$
\end{prop}
\begin{proof}
We recall that, under $\ov{\P}$, the random variables $(\tau_{x_i})_{i \in \N}$ are independent and distributed according to $\mu_0$. The family $(\tau_{x\de(n)})_{n \in \N}$ is the subsequence made of those elements whose value exceeds $\delta \eps^{-1/\alpha}$. Hence, for any $n \in \N$, the law of $\tau_{x\de(n)}$ is the one of $\tau_0$ conditionned on being larger than $\delta \eps^{-1/\alpha}$, and we obtain, for any $x \ge \delta$~:
$$
\P[\eps^{1/\alpha} \tau_{x\de(n)} \ge x] = \frac{\P[\eps^{1/\alpha} \tau_0 \ge x]}{\P[\eps^{1/\alpha} \tau_0 \ge \delta]},
$$
which, according to (\ref{regvar}), converges to $\delta^\alpha/x^\alpha$ as $\eps$ goes to $0$.
\end{proof}

We write $D([0,t],\R)$ for the space of cadlag functions from $[0,t]$ to $\R$. For a definition of the $M_1$ distance on $D([0,t],\R)$, we refer to \cite[(3.3.4)]{whitt} (or equivalently, \cite[(12.3.8)]{whitt}). With a slight abuse of notation, we will not distinguish between a process and its restriction on $[0,t]$. 

\begin{prop}
\label{distM1}
\begin{enumerate}
\item
For any $n \in \N$ and any $t > 0$, the $M_1$ distance on $D([0,t],\R)$ between $l\de(n,\cdot)$ and the step function
$$
l\de(n) \ \1_{\{ \cdot \ge \eps T\de(n) \}}
$$
converges to $0$ in probability under $\ov{\P}$ as $\eps$ tends to $0$. 
\item For any $t > 0$, the $M_1$ distance on $D([0,t],\R)$ between $H\de$ and $\mathcal{H}\de$ (resp. between $L\de$ and $\mathcal{L}\de$) converges to $0$ in probability under $\ov{\P}$ as $\eps$ tends to~$0$.
\end{enumerate}
\end{prop}
\begin{proof}
Let $\eta > 0$ be some small parameter. We begin by showing that the difference between $l\de(n,\cdot)$ evaluated at time $\eps T\de(n) + \eta$ and its limit is small, in the sense that
\begin{equation}
\label{M1droite}
\ov{\E} \left[ l\de(n) - l\de\left(n,\eps T\de(n) + \eta\right) \right]  \xrightarrow[\eps \to 0]{} 0.
\end{equation}
One can rewrite the left hand side above as
\begin{equation}
\label{M1droiteetap1}
\ov{\E}\left[ \int_{T\de(n) + \eps^{-1} \eta}^{+\infty} \1_{\{\hX_s = x\de(n)\}} \ \d s \right].
\end{equation}
Recall that $x\de(n)$ is a site discovered at time $T\de(n)$, hence it is a neighbour of $X_{T\de(n)}$, and we can bound the integrand above by~:
$$
\1_{\{ |\hX_s-\hX_{T\de(n)}|\le 1 \}}.
$$
Using this together with the Markov property at time $T\de(n)$ leads one to bound the term in (\ref{M1droiteetap1}) by~:
\begin{equation*}
\ov{\E} \left[ \EEt_{\hX_{T\de(n)}}\left[ \int_{\eps^{-1} \eta}^{+\infty} \1_{\{|\hX_s - \hX_0| \le 1\}} \ \d s \right] \right] .
\end{equation*}
But as we recall in Proposition~\ref{diagonaldecay} of the Appendix, there exists $C$ such that for any $s > 0$~:
$$
\sup_{x,y} \PPt_x[X_s = y]  \le \frac{C}{s^{d/2}},
$$
from which (\ref{M1droite}) follows. In particular, this implies that the probability of the event~:
\begin{equation}
\label{M1event}
l\de(n) - l\de\left(n,\eps T\de(n) + \eta\right) \le \eta
\end{equation}
converges to $1$ as $\eps$ goes to $0$. On this event, the increasing process $l\de(n,\cdot)$ is constant equal to $0$ up to time $\eps T\de(n)$, and reaches a value close to its limit by $\eta$ at time $(\eps T\de(n) + \eta)$. From this observation, it is not hard to construct parametrizations of the completed graphs (as defined in \cite[(3.3.3)]{whitt}) of $l\de(n,\cdot)$ and of the step function that show the $M_1$ distance on $D([0,t],\R)$ to be smaller than $2 \eta$, provided $\eps T\de(n)$ does not lie in $[t-\eta,t]$. By Proposition~\ref{poisson}, the probability that such an event happens is as close to $0$ as desired, thus ending the proof of the first part of the proposition.

Let us now turn to the second part of the proposition. We recall that $\mathcal{H}\de$ was defined in~(\ref{defmclH}). Using the previous result, together with the fact that the random variable $\eps^{1/\alpha} \tau_{x\de(n)}$ is tight by Proposition~\ref{taulim}, we obtain that the $M_1$ distance between 
$$
l\de(n,\cdot) \ \eps^{1/\alpha} \tau_{x\de(n)}
$$
on one hand, and 
$$
l\de(n) \ \eps^{1/\alpha} \tau_{x\de(n)} \ \1_{\{\cdot \ge \eps T\de(n)\}}
$$
on the other, goes to $0$ in probability as $\eps$ tends to $0$. Moreover, because of Proposition~\ref{poisson}, the number of $n$'s such that $\eps T\de(n)$ belongs to $[0,t]$ is bounded in probability. Hence, when considering $H\de$ in (\ref{Hderewrit}), one can restrict the sum to a finite number of terms, and then apply the above observation to each of the terms, thus proving the proposition. The same proof applies as well to $L\de$ and $\mathcal{L}\de$, using the representation in (\ref{Lderewrit}) and the fact that the function~$h$ is bounded.
\end{proof}

%
%
%
%
%
%

\section{The environment around a trap}
\label{s:envatrap}
\setcounter{equation}{0}
Consider the environment around the $n^{\text{th}}$ deep trap, $\theta_{x\de(n)} \ \tau$. We have already seen in Proposition~\ref{taulim} the convergence in law, after proper scaling, of $\left( \theta_{x\de(n)} \ \tau \right)_0 = \tau_{x\de(n)}$. We would like to gain information about the other coordinates of $\theta_{x\de(n)} \ \tau$. For any $z \neq 0$, let $\tau\de(n,z)$ be defined by
$$
\tau\de(n,z) = \left( \theta_{x\de(n)} \ \tau \right)_z = \tau_{x\de(n)+z}.
$$
For convenience, we write $\tau\de(n)$ for the family $\left(\tau\de(n,z)\right)_{z \neq 0}$, and may call $\tau\de(n)$ the \emph{environment around the} $n^{th}$ \emph{deep trap}. We insist that this environment has not any value asigned at the origin.

We will show that $\tau\de(n)$ converges in law (for the product topology) as $\eps$ goes to $0$. The next proposition is a first step in this direction.

\begin{prop}
\label{tightenv}
For any integer $n$, the family of random variables $(\tau\de(n))_{\eps > 0}$ is tight under $\ov{\P}$.
\end{prop}
\begin{proof}
Let $z \in \Z^d \setminus \{ 0 \}$. It suffices to show that, for any $\eta > 0$, there exists $M > 0$ such that, for $\eps$ small enough,
\begin{equation}
\label{tightcrit}
\ov{\P}\left[ \tau_{x\de(n) + z} \ge M \right] \le \eta.
\end{equation}
We say that $x \in \Z^d$ is \emph{atypical} if it is a deep trap and the depth of the site $(x+z)$ exceeds $M$~:
$$
\eps^{1/\alpha} \tau_x \ge \delta \quad \text{and} \quad \tau_{x+z} \ge M.
$$
The event appearing in the left hand-side of (\ref{tightcrit}) can be rephrased as saying that $x\de(n)$ is an atypical trap.
We say that $x \in \Z^d$ is \emph{uncommon} if it is atypical, or if $(x-z)$ is atypical. Finally, for a subset $\Gamma \subset \Z^d$, we say that $x$ is \emph{uncommon regardless of} $\Gamma$ if one can infer that $x$ is uncommon without considering sites inside $\Gamma$, i.e. if one of the two following conditions occur~:
$$
x \text{ is atypical and } \{x,x+z\} \cap \Gamma = \emptyset, \quad \text{or} \quad (x-z) \text{ is atypical and } \{x-z,x\} \cap \Gamma = \emptyset.
$$
Let us assume momentarily the validity of the following lemma, and see how it enables us to show the proposition.
\begin{lem}
\label{atyp}
If $x\de(n)$ is atypical, then there exists $k \le r\de(n)$ such that $x_k$ is uncommon regardless of $\{x_1,\ldots,x_{k-1}\}$.
\end{lem}
We saw in the proof of Proposition~\ref{poisson} that the random variables $\eps r\de(n)$ converge in law as $\eps$ tends to $0$. Therefore, one can find a constant $C_r$ such that the probability of the event 
\begin{equation}
\label{rdeCr}
\eps r\de(n) \le C_r
\end{equation}
is as close to $1$ as desired when $\eps$ is small. On this event, using the result of the lemma, the fact that $x\de(n)$ is atypical implies that there exists $k \le \eps^{-1} C_r$ such that $x_k$ is uncommon regardless of $\{x_1,\ldots,x_{k-1}\}$. The probability of this event is bounded by
\begin{equation}
\label{sumatypicals}
\sum_{k = 1}^{\eps^{-1} C_r} \ov{\P}[x_k \text{ is uncommon regardless of } \{x_1,\ldots,x_{k-1}\}].
\end{equation}
We now proceed to evaluate the generic term of this sum. 
We will condition on the trajectories up to the discovery of the $k^\text{th}$ site. We refer to the proof of Proposition~\ref{tauiid} for the definitions of $\gamma^{\leftarrow}$ (where $\gamma$ is a path), the set of paths $E_k$ and the meaning of the event that we write ``$Y = \gamma$''.
\begin{equation*}
\begin{split}
& \ov{\P}[x_k \text{ uncommon regardless of } \{x_1,\ldots,x_{k-1}\}] \\
& \qquad  =  \sum_{\gamma \in \E_k} \ov{\P}[Y = \gamma, \ x_k(\gamma) \text{ uncommon regardless of } \{x_1(\gamma),\ldots,x_{k-1}(\gamma)\}] \\
& \qquad =  \sum_{\gamma \in \E_k} \E\left[ \PPt_0[Y = \gamma] , \ x_k(\gamma) \text{ uncommon regardless of } \{x_1(\gamma),\ldots,x_{k-1}(\gamma)\} \right].
\end{split}
\end{equation*}
Moreover, the probability $\PPt_0[Y = \gamma]$ depends only on $\mcD(\gamma^\leftarrow)$, while the event
$$
x_k(\gamma) \text{ is uncommon regardless of } \{x_1(\gamma),\ldots,x_{k-1}(\gamma)\}
$$
has been constructed in order to depend only on sites outside $\mcD(\gamma^\leftarrow)$. Due to the fact that $\P$ is a product measure, it comes that
\begin{equation*}
\begin{split}
& \ov{\P}[x_k \text{ is uncommon regardless of } \{x_1,\ldots,x_{k-1}\}] \\
& \qquad = \sum_{\gamma \in \E_k} \ov{\P}[Y = \gamma] \ \P[x_k(\gamma) \text{ is uncommon regardless of } \{x_1(\gamma),\ldots,x_{k-1}(\gamma)\}] \\
& \qquad \le \sum_{\gamma \in \E_k} \ov{\P}[Y = \gamma] \ \P[x_k(\gamma) \text{ is uncommon}],
\end{split}
\end{equation*}
Translation invariance of the measure $\P$ implies that in fact, 
$$
\P[x \text{ is uncommon}]
$$
does not depend on $x$. We have thus shown that the sum in (\ref{sumatypicals}) is bounded by
\begin{equation}
\label{tightenvconcl}
\eps^{-1} C_r \P[0 \text{ is uncommon}] \le 2 \eps^{-1} C_r \P[0 \text{ is atypical}],
\end{equation}
a term which should be uniformly small as $\eps$ goes to $0$, when $M$ is chosen large enough. It is easily seen to be so noting that
$$
\P[0 \text{ is atypical}] = \P[\eps^{1/\alpha} \tau_0 \ge \delta] \ \P[\tau_z \ge M],
$$
and that $\eps^{-1} \P[\eps^{1/\alpha} \tau_0 \ge \delta]$ is bounded as $\eps$ goes to $0$, while $\P[\tau_z \ge M]$ can be made arbitrarily small by choosing $M$ large enough.
\end{proof}
\begin{proof}[Proof of Lemma~\ref{atyp}]
Let $k_0$ be defined by
$$
k_0 = \min \{ k : x_k \text{ is uncommon} \}.
$$
If $x\de(n)$ is atypical, then in particular it is uncommon, hence on this event, $k_0$ is finite and smaller than $r\de(n)$. Two situations may occur. If $x_{k_0}$ is atypical, then $x_{k_0}+z$ is uncommon, hence does not belong to $ \{x_1,\ldots, x_{k_0-1}\}$, so that
$$
\{x_{k_0},x_{k_0}+z\} \cap \{x_1,\ldots, x_{k_0-1}\} = \emptyset.
$$
In this case, $x_{k_0}$ is indeed uncommon regardless of $\{x_1,\ldots, x_{k_0-1}\}$. On the other hand, if it is $x_{k_0}-z$ that is atypical, then in particular it is uncommon, hence it does not belong to $\{x_1,\ldots, x_{k_0-1}\}$, and the intersection
$$
\{x_{k_0}-z,x_{k_0}\} \cap \{x_1,\ldots, x_{k_0-1}\}
$$
is empty, a fact from which the conclusion follows as well.
\end{proof}

We will show in the next proposition that the asymptotic behaviour of $l\de(n)$ can be inferred from the one of $\tau\de(n)$. Let $G(\tau)$ be the Green function at the origin~:
$$
G(\tau) = \EEt_0 \left[ \int_0^{+\infty} \1_{\{ \hX_s = 0 \}} \ \d s\right].
$$
Let $e\de(n)$ be such that
\begin{equation}
\label{defeden}
l\de(n) = e\de(n) G\left(\theta_{x\de(n)} \ \tau\right).
\end{equation}
From the fact that, with high probability, the site $x\de(n)$ is visited by the random walk, one can easily derive that $e\de(n)$ converges in law to an exponential random variable of parameter one. 

We recall from (\ref{defmclT}) that we write $\mathcal{T}_x$ for the hitting time of $x$.
\begin{prop}
\label{visittrap}
\begin{enumerate}
\item
The probability that the site $x\de(n)$ is visited by the random walk goes to $1$ as $\eps$ goes to $0$~:
$$
\ov{\P}\left[\mathcal{T}_{x\de(n)} < \infty\right] \xrightarrow[\eps \to 0]{} 1.
$$
\item For any $u \ge 0$ and any $x \in \Z^d$, one has
\begin{multline}
\label{edencalcul}
\PPt_0[e\de(n) \ge u, \ x\de(n) = x, \ \mathcal{T}_{x\de(n)} <  \infty]  \\ = e^{-u} \ \PPt_0[x\de(n) = x, \ \mathcal{T}_{x\de(n)} <  \infty].
\end{multline}
\item
As $\eps$ goes to $0$, the random variable $e\de(n)$ converges in law under $\ov{\P}$ to an exponential random variable of parameter $1$.
\end{enumerate}
\end{prop}
Before turning to the proof, we introduce some notation. Let $q^\tau(x,y)$ be the probability for the walk starting from $x$ to make its first jump on the site $y$. When $x,y \in \Z^d$ are neighbours, we write
\begin{equation}
\label{defsigma}
\sigma^\tau(x,y) = \sum_{\substack{z \sim x \\ z \neq y}} (\tau_z)^a.
\end{equation}
Moreover, for any $x,y \in \Z^d$ we write $x \approx y$ if $x$ is a neighbour or a second neighbour of $y$.
\begin{proof}[Proof of Proposition~\ref{visittrap}]
From Proposition~\ref{tightenv}, we know that the probability of the event
\begin{equation}
\label{eventboundedapprox}
\forall y \approx x\de(n) \quad \tau_y \le M
\end{equation}
can be made as close to $1$ as desired, by choosing $M$ large enough. Let us assume that the position of the walk at the time of discovery of $x\de(n)$ is $x$. In particular, $x$ is a neighbour of $x\de(n)$, and the probability (for a fixed environment) that from $x$, the walk jumps to $x\de(n)$ is given by
\begin{equation}
\label{probasaut}
q^\tau(x,x\de(n)) = \frac{\left(\tau_{x\de(n)}\right)^a}{\sum_{z \sim x} (\tau_z)^a} = \left( 1 +  \frac{\sigma_\tau(x,x\de(n))}{\left(\tau_{x\de(n)}\right)^a} \right)^{-1}.
\end{equation}
On the event (\ref{eventboundedapprox}), the random variable $\sigma_\tau(x,x\de(n))$ is bounded (uniformly over $x$), while $\tau_{x\de(n)}$ is larger than $\eps^{-1/\alpha} \delta$. Hence, the quantity in (\ref{probasaut}) goes to $1$ in probability, which proves the first part of the proposition.

Let us now consider the second part. We have~:
\begin{multline*}
 \PPt_0[e\de(n) \ge u, \ x\de(n) = x, \ \mathcal{T}_{x\de(n)} <  \infty] \\
 = \PPt_0\left[\int_{\mathcal{T}_x}^{\infty} \1_{\{ \hX_s = x \}} \ \d s \ge u G(\theta_x \ \tau), \ x\de(n) = x, \ \mathcal{T}_x < \infty \right].
\end{multline*}
The Markov property at time $\mathcal{T}_x$ enables us to rewrite it as
\begin{equation}
\label{expcalcul}
\PPt_x\left[\int_0^{\infty} \1_{\{ \hX_s = x \}} \ \d s \ge u G(\theta_x \ \tau) \right] \ \PPt_0[x\de(n) = x, \ \mathcal{T}_x < \infty].
\end{equation}
Starting from $x$, the total time spent on site $x$ is an exponential random variable of parameter $G(\theta_x \ \tau)$. Hence, the first term in (\ref{expcalcul}) is equal to $e^{-u}$, and we obtain the announced claim. 

The third part of the proposition is a direct consequence of the first two. Indeed, summing over all $x \in \Z^d$ and integrating over the environment in equation (\ref{edencalcul}), one obtains that, conditionally on $\mathcal{T}_{x\de(n)} < \infty$, the random variable $e\de(n)$ is distributed under $\ov{\P}$ as an exponential random variable of parameter $1$. The result is then obtained using the fact that the probability of the event $\mathcal{T}_{x\de(n)} < \infty$ goes to $1$ as $\eps$ goes to $0$. 
\end{proof}
In Proposition~\ref{compG} of the Appendix, we show that the limit
\begin{equation}
\label{limitG}
\ov{G}\left( (\tau_z)_{z \neq 0} \right) = \lim_{\tau_0 \to +\infty} G(\tau)
\end{equation}
exists. The next proposition claims that $G\left(\theta_{x\de(n)} \ \tau\right)$ is well approximated by $\ov{G}\left(\tau\de(n)\right)$.
\begin{prop}
\label{lde}
The difference
$$
\left|G\left(\theta_{x\de(n)} \ \tau\right) - \ov{G}\left(\tau\de(n)\right) \right|
$$
converges to $0$ in $\ov{\P}$-probability as $\eps$ tends to $0$.
\end{prop}
\begin{proof}
We recall that we denote by $q^\tau(x,y)$ the probability for the walk starting from $x$ to jump to the site $y$. Proposition~\ref{compG} of the Appendix states that, for any environment $\tau$~:
$$
0 \le G(\tau) - \ov{G}(\tau) \le G(\tau) \left(1 - \min_{y \sim 0} q^\tau(y,0)^2 \right),
$$
and moreover, that $G(\tau)$ is uniformly bounded. Hence, in order to prove the claim, it suffices to show that
$$
\min_{y \sim x\de(n)} q^\tau(y,x\de(n))^2
$$
converges to $1$ in $\ov{\P}$-probability as $\eps$ tends to $0$. This fact has in fact already been shown to hold during the proof of Proposition~\ref{visittrap}.
\end{proof}

We now precise the particular form of the limits of $L\de$ and $H\de$. From Propositions~\ref{poisson}, \ref{taulim}, \ref{tightenv} and~\ref{visittrap}, we know that the joint distribution of 
\begin{equation}
\label{bigjoint}
\left(T\de(n), \eps^{1/\alpha} \tau_{x\de(n)}, \tau\de(n) ,e\de(n)\right)_{n \in \N}
\end{equation}
is tight under $\ov{\P}$. Let $(\eps_k)_{k \in \N}$ be a sequence on which the joint law of (\ref{bigjoint}) converges. Possibly enlarging the probability space, we assume that there exist random variables which are distributed acording to this limit law under $\ov{\P}$, and which we denote by
\begin{equation}
\label{bigjointlimit}
\left(T_\delta(n), \tau^\circ_\delta(n), \tau_\delta(n) , e_\delta(n) \right)_{n \in \N}.
\end{equation}
Hence, we assume that the following convergence holds~:
\begin{multline}
\label{convbigjoint}
\left(T\dek(n), \eps_k^{1/\alpha} \tau_{x\dek(n)}, \tau\dek(n) ,e\dek(n)\right)_{n \in \N} \\ \xrightarrow[k \to +\infty]{\text{law}} \left(T_\delta(n), \tau^\circ_\delta(n), \tau_\delta(n) , e_\delta(n) \right)_{n \in \N}.
\end{multline}
\begin{prop}
\label{limitHform}
Let $(\eps_k)$ be a sequence such that (\ref{convbigjoint}) holds. The laws of the processes $H\dek$ and $L\dek$ converge, respectively, to the ones of $H_\delta$ and $L_\delta$, defined by~:
\begin{equation}
\label{convH}
H_\delta(\cdot) = \sum_{n = 1}^{+\infty} e_\delta(n) \ov{G}(\tau_\delta(n)) \tau^\circ_\delta(n)  \1_{\{\cdot \ge T_\delta(n)\}},
\end{equation}
\begin{equation}
\label{convL}
L_\delta(\cdot) = \sum_{n = 1}^{+\infty} e_\delta(n) (\ov{G} h)(\tau_\delta(n)) \1_{\{\cdot \ge T_\delta(n)\}},
\end{equation}
this convergence holding both for the $M_1$ topology and in the sense of finite-dimen\-sional distributions.
\end{prop}
\noindent \textbf{Remark.} One can check by a careful reading of the proof below that the convergence of the joint law of $(H\dek, L\dek)$ holds. Yet, as we are not interested in this stronger fact, Proposition~\ref{limitHform} should be understood in the sense of separate convergence of the laws of $H\dek$ and $L\dek$.
\begin{proof}
We begin by showing that, if convergence holds for the $M_1$ topology, then it also holds in the sense of finite-dimensional distributions. For $t > 0$, let $t_1,\ldots,t_n \in [0,t]$, and consider the projection 
$$
\pi : 
\left\{
\begin{array}{ccc}
D([0,t],\R) & \to & \R^{n} \\
Z & \mapsto & (Z_{t_1},\ldots, Z_{t_n}).
\end{array}
\right.
$$
Proposition~\ref{poisson} ensures that almost surely, the sequence of jumps $(T_\delta(n))_{n \in \N}$ does not intersect the set $\{t_1,\cdots, t_n\}$.The limit processes $H_\delta$ and $L_\delta$ (restricted to $[0,t]$) are therefore almost surely inside the set of continuity points of $\pi$, and the claim follows using the continous mapping theorem \cite[Theorem~2.7]{bill}.

For $t > 0$, let us prove the convergence of $H\dek$ to $H_\delta$ for the $M_1$ topology on $D([0,t],\R)$. We recall from (\ref{defmclH}) and (\ref{defeden}) that
\begin{equation}
\label{rappeldefmclH}
\mathcal{H}\dek(\cdot) = \sum_{n = 1}^{+\infty} e\dek(n) G\left(\theta_{x\dek(n)} \ \tau\right) \ \eps_k^{1/\alpha} \tau_{x\dek(n)} \ \1_{\{\cdot \ge \eps_k T\dek(n)\}}.
\end{equation}
By Proposition~\ref{distM1}, it is enough to show that $\mathcal{H}\dek$ converges in distribution to $H_\delta$. Moreover, by Proposition~\ref{lde}, we may as well replace $\mathcal{H}\dek$ by the process
\begin{equation}
\label{remplacemclH}
\sum_{n = 1}^{+\infty} e\dek(n) \ov{G}\left(\tau\dek(n)\right) \ \eps_k^{1/\alpha} \tau_{x\dek(n)} \ \1_{\{\cdot \ge \eps_k T\dek(n)\}}.
\end{equation}

By Skorokhod's representation theorem \cite[Theorem~6.7]{bill}, there exist random variables 
$$
\left(\td{T}\dek(n), \eps_k^{1/\alpha} \td{\tau}_{x\dek(n)}, \td{\tau}\dek(n) ,\td{e}\dek(n)\right)_{n \in \N}
$$
that, for fixed $k \in \N$, have the same joint law as 
$$
\left(T\dek(n), \eps_k^{1/\alpha} \tau_{x\dek(n)}, \tau\dek(n) ,e\dek(n)\right)_{n \in \N},
$$
and converge almost surely, as $k$ goes to infinity, to other random variables that we write
\begin{equation}
\label{tdbigjointlimit}
\left(\td{T}_\delta(n), \td{\tau}^\circ_\delta(n), \td{\tau}_\delta(n) , \td{e}_\delta(n) \right)_{n \in \N}.
\end{equation}
Naturally, the random variables in (\ref{tdbigjointlimit}) have the same joint law as the ones in (\ref{bigjointlimit}).
Let $\td{\mathcal{H}}\dek$ be the process defined by
$$
\td{\mathcal{H}}\dek(s) = \sum_{n = 1}^{+\infty} \td{e}\dek(n) \ov{G}\left(\td{\tau}\dek(n)\right) \ \eps_k^{1/\alpha} \td{\tau}_{x\dek(n)} \ \1_{\{s \ge \eps_k \td{T}\dek(n)\}}.
$$
The process $\td{\mathcal{H}}\dek$ has the same law as the one defined in (\ref{remplacemclH}). We will show that it converges almost surely (for Skorokhod's $M_1$ topology) to the process $\td{H}_\delta$ defined by
$$
\td{H}_\delta(s) = \sum_{n = 1}^{+\infty} \td{e}_\delta(n) \ov{G}(\td{\tau}_\delta(n)) \td{\tau}^\circ_\delta(n)  \1_{\{s \ge \td{T}_\delta(n)\}}.
$$
This result would prove the proposition, as it is clear that $\td{H}_\delta$ and $H_\delta$ have the same distribution. 

Because of Proposition~\ref{poisson}, we know that jump instants are almost surely distinct in the limit, hence for any $s$ that does not belong to $\{\td{T}_\delta(n), n \in \N\}$, one has
\begin{equation}
\label{cvponctuel}
\td{\mathcal{H}}\dek(s) \xrightarrow[k \to + \infty]{} \td{H}_\delta(s).
\end{equation}
We can then apply the criterion for $M_1$ convergence given in \cite[Theorem 12.5.2 (iii)]{whitt}, noting that the oscillation function appearing in this criterion is zero for increasing functions.

The same proof applies as well for $L\dek$. Indeed, Proposition~\ref{distM1} ensures that one can approximate the process by $\mathcal{L}\dek$ in (\ref{defmclL}), which, thanks to Proposition~\ref{lde}, is in turn well approximated by
$$
\sum_{n = 1}^{+\infty} e\dek(n) \ov{G}\left(\tau\dek(n)\right) \ h\left(\theta_{x\dek(n)} \ \tau \right) \ \1_{\{\cdot \ge \eps_k T\dek(n)\}}.
$$
As the function $h$ is such that $h(\tau)$ does not depend on $\tau_0$, one has
$$
h\left(\theta_{x\dek(n)} \ \tau \right) = h\left(\tau\dek(n)\right),
$$
and the rest of the proof follows.
\end{proof}

Before being able to show that the environment around a trap has a unique possible limit law, and to describe it explicitly, we need to show independence between $e_\delta(n)$ and $\tau_\delta(n)$. Having in mind that we will need to study the jumps of $H\de$ as well, we will show the following stronger result.
\begin{prop}
\label{trioindep}
Let $(\eps_k)$ be a sequence such that (\ref{convbigjoint}) holds.
The random variables $\tau^\circ_\delta(n)$, $\tau_\delta(n)$ and $e_\delta(n)$ are independent.
\end{prop}
\begin{proof}
Let $f_1,f_2 : \R \to \R$ and $f_3 : \R^{\Z^d \setminus \{0\}} \to \R$ be three bounded continuous functions. We are interested in
$$
\ov{\E}\left[ f_1(e_\delta(n)) f_2(\tau^\circ_\delta(n)) f_3(\tau_\delta(n)) \right].
$$
Because of part 1 of Proposition~\ref{visittrap}, this expectation can be obtained as the limit as $k$ tends to infinity of
$$
\ov{\E}\left[ f_1(e\dek(n)) f_2(\eps_k^{1/\alpha} \tau_{x\dek(n)}) f_3(\tau\dek(n)) \ \1_{\{ \mathcal{T}_{x\dek(n)} < \infty \}} \right].
$$
Observe that, in a fixed environment, $f_2(\eps_k^{1/\alpha} \tau_{x\dek(n)}) f_3(\tau\dek(n))$ is a function of $x\dek(n)$ only. Using the second part of Proposition~\ref{visittrap}, we obtain~:
\begin{multline*}
\EEt_0\left[ f_1(e\dek(n)) f_2(\eps_k^{1/\alpha} \tau_{x\dek(n)}) f_3(\tau\dek(n)) \ \1_{\{ \mathcal{T}_{x\dek(n)} < \infty \}} \right] \\
= \int f_1(x) e^{-x}  \d x \ \EEt_0\left[ f_2(\eps_k^{1/\alpha} \tau_{x\dek(n)}) f_3(\tau\dek(n)) \ \1_{\{ \mathcal{T}_{x\dek(n)} < \infty \}} \right].
\end{multline*}
We are thus left with the study of
\begin{equation}
\label{indeptautau0}
\ov{\E}\left[ f_2(\eps_k^{1/\alpha} \tau_{x\dek(n)}) f_3(\tau\dek(n)) \right].
\end{equation}
We partition according to the events $\{x\dek(n) = x\}$, for $x \in \Z^d$~:
$$
\sum_{x \in \Z^d} \ov{\E}\left[ f_2(\eps_k^{1/\alpha} \tau_{x}) f_3(\theta_x \ \tau) \ \1_{\{x\dek(n) = x\}} \right].
$$
We recall that we say that a site $x$ is deep if $\eps^{1/\alpha} \tau_x \ge \delta$. Noting that on the event $x\dek(n) = x$, the site $x$ is deep, one can rewrite the generic term of the sum above as~:
\begin{multline}
\label{indeptautau}
\ov{\E}\left[ f_2(\eps_k^{1/\alpha} \tau_{x}) f_3(\theta_x \ \tau) \ \1_{\{x\dek(n) = x\}} \ \big| \ x \text{ deep} \right] \ \P[x \text{ deep} ] \\
= \E \left[ f_2(\eps_k^{1/\alpha} \tau_{x}) f_3(\theta_x \ \tau) \ \PPt_0[x\dek(n) = x] \ \big| \ x \text{ deep} \right] \ \P[x \text{ deep} ]     .
\end{multline}
Let us write $\mathcal{A}$ for the event ``there are exactly $n-1$ deep traps discovered before the walk discovers site $x$''. Conditionally on the fact that the site $x$ is deep, we have
\begin{equation}
\label{equalityprobab}
\PPt_0[x\dek(n) = x] = \PPt_0[\mathcal{A}].
\end{equation}
As the law of the trajectory up to the instant of discovery of site $x$ does not depend on $\tau_x$, the probability $\PPt_0[\mathcal{A}]$ does not depend on $\tau_x$. Moreover, by the definition of $f_3$, the quantity $f_3(\theta_x \ \tau)$ does not depend on $\tau_x$ either. Using the fact that the measure $\P$ conditioned on the event that $x$ is a deep trap is a product measure, we obtain that the first expectation appearing in (\ref{indeptautau}) is equal to
$$
\E \left[ f_2(\eps_k^{1/\alpha} \tau_{x})  \ \big| \ x \text{ deep} \right] \E\left[f_3(\theta_x \ \tau) \ \PPt_0[\mathcal{A}]  \ \big| \ x \text{ deep} \right].
$$
Using (\ref{equalityprobab}) once more, we observe that
\begin{eqnarray*}
\E\left[f_3(\theta_x \ \tau) \ \PPt_0[\mathcal{A}]  \ \big| \ x \text{ deep} \right] &  = & \E\left[f_3(\theta_x \ \tau) \ \PPt_0[x\dek(n) = x]  \ \big| \ x \text{ deep} \right] \\
& = & \ov{\E}\left[ f_3(\theta_x \ \tau) \ \1_{\{x\dek(n) = x\}} \ \big| \ x \text{ deep} \right] .
\end{eqnarray*}
As a consequence, the product in (\ref{indeptautau}) is equal to
$$
\ov{\E}\left[ f_2(\eps_k^{1/\alpha} \tau_{x}) \ \big| \ x \text{ deep} \right] \ \ov{\E}\left[ f_3(\theta_x \ \tau) \ \1_{\{x\dek(n) = x\}} \ \big| \ x \text{ deep} \right] \ \P[x \text{ deep} ].
$$
The first expectation does not depend on $x$. The two last terms can be merged together to make the conditioning disappear. Summing over all $x \in \Z^d$, we recover the expectation in (\ref{indeptautau0}), which is therefore equal to~:
$$
\ov{\E}\left[ f_2(\eps_k^{1/\alpha} \tau_{0}) \ \big| \ 0 \text{ deep} \right] \ \ov{\E}\left[ f_3(\tau\dek(n)) \right].
$$
This proves the independence of the random variables we were interested in, taking the limit $k \to +\infty$.
\end{proof}

We are now able to show that the environment around the first deep trap converges in distribution. We have already seen in Proposition~\ref{tightenv} that $\tau\de(1)$ is tight. Hence, what we need to see is that there is only one possible limit law.
\begin{prop}
\label{convenv}
Let $(\eps_k)$ be a sequence such that (\ref{convbigjoint}) holds. The law of $\tau_\delta(1)$ is characterized by the fact that, for any test function $h$~:
\begin{equation}
\label{charactenv}
c \ \ov{\E}\left[ (\ov{G} h)(\tau_\delta(1))  \right] =  \E[h(\tau)],
\end{equation}
where $c$ is the constant appearing in Proposition~\ref{range}. In particular, the law of $\tau_\delta(1)$ does not depend on the sequence $(\eps_k)$, nor on $\delta$.
\end{prop}
\begin{proof}
The proof uses the fact that the expectation of $L\de(t)$ is easy to compute. Hence, we need to change the convergence in distribution in Proposition~\ref{limitHform} into convergence in the mean. This is done by the following lemma, which we momentarily admit.
\begin{lem}
\label{unifint}
For any $t \ge 0$, the family of random variables $(L\de(t))_{\eps > 0}$ is uniformly integrable.
\end{lem}
From the definition of $L\de$ in (\ref{defLde}), and using the stationarity of the environment viewed by the particle, we have
$$
\ov{\E}\left[ L\de(t) \right] = \eps^{-1} t \E[h(\tau) \ \1_{\eps^{1/\alpha} \tau_{0} \ge \delta} ].
$$
Because $h(\tau)$ does not depend on $\tau_0$, the expectation above is equal to
$$
\E[h(\tau)] \ \P[{\eps^{1/\alpha} \tau_{0} \ge \delta} ].
$$
Using the tail behaviour of $\tau_0$ given in (\ref{regvar}), we finally obtain that
\begin{equation}
\label{ELt0}
\ov{\E}\left[ L\de(t) \right] \xrightarrow[\eps \to 0]{} \frac{t}{\delta^\alpha} \E[h(\tau)].
\end{equation}
We now compute the expectation of the limit $L_\delta(t)$ (see (\ref{convL})). 
%
Note first that the quantity
$$
\ov{\E}\left[ \sum_{n = 2}^{+ \infty} \1_{\{ t \ge T_\delta(n) \}} \right]
$$
is the expected number of points from $(T_\delta(n))_{n \in \N \setminus \{1\}}$ that fall within $[0,t]$. Because of Proposition~\ref{poisson}, it is $O(t^2)$ when $t$ goes to $0$. Moreover, as given by Proposition~\ref{indepincrL}, the process $L_\delta$ is a subordinator. In particular, the jump instants are independent from the heights of the jumps, so that we have
$$
\ov{\E}\left[ L_\delta(t) \right] = \ov{\E}[e_\delta(1) (\ov{G} h)(\tau_\delta(1))] \ \ov{\P}[t \ge T_\delta(1)] + O(t^2) \qquad (t \to 0).
$$
Here, we used the fact that, as the function $h$ takes values in $(0,+\infty)$, the quantity $e_\delta(1) (\ov{G} h)(\tau_\delta(1))$ is non-zero, and there is indeed a jump at $T_\delta(1)$.

We saw in Proposition~\ref{visittrap} that $e_\delta(1)$ is an exponential random variable of parameter $1$, and in Proposition~\ref{trioindep} that it is independent from $\tau_\delta(1)$, hence
$$
\ov{\E}[e_\delta(1) (\ov{G} h)(\tau_\delta(1))] = \ov{\E}[(\ov{G} h)(\tau_\delta(1))].
$$
From Proposition~\ref{poisson}, we know that
$$
\ov{\P}[t \ge T_\delta(1)] = \frac{c t}{\delta^\alpha} + O(t^2),
$$
so we obtain~:
$$
\ov{\E}\left[ L_\delta(t) \right] = \frac{c t}{\delta^\alpha} \ov{\E}[(\ov{G} h)(\tau_\delta(1))] + O(t^2).
$$
Comparing this with (\ref{ELt0}) leads to (\ref{charactenv}). Let us see that this relation characterizes the law of $\tau_\delta(1)$. First, one can check that the relation (\ref{charactenv}) still holds without the restriction that the function $h$ should have values only in $(0,+\infty)$. Let $f$ be a positive bounded continuous function, such that $f(\tau)$ does not depend on $\tau_0$. For $\eta > 0$, we define $h$ as
$$
h(\tau) = \left|
\begin{array}{ll}
\ov{G}(\tau)^{-1} f(\tau) & \text{if } \ov{G}(\tau) \ge \eta \\
0 & \text{otherwise.}
\end{array}
\right.
$$
Then, from (\ref{charactenv}), one has~:
$$
c \ \ov{\E}\left[f(\tau_\delta(1)) \1_{\{ \ov{G}(\tau_\delta(1)) \ge \eta \}} \right] =  \E\left[\ov{G}(\tau)^{-1} f(\tau) \1_{\{ \ov{G}(\tau) \ge \eta \}}\right].
$$
Taking the limit as $\eta$ tends to $0$, and using monotone convergence theorem, we obtain~:
\begin{equation}
\label{charactbis}
c \ \ov{\E}\left[f(\tau_\delta(1))\right] =  \E\left[\ov{G}(\tau)^{-1} f(\tau)\right].
\end{equation}
Being valid for any positive bounded continuous function, equation (\ref{charactbis}) determines the law of $\tau_\delta(1)$.
\end{proof}
\begin{proof}[Proof of Lemma~\ref{unifint}]
We will use the following upper bound on $L\de(t)$ (see (\ref{Lderewrit}))~:
\begin{equation}
\label{upperboundLde}
L\de(t) \le \|h\|_\infty \sum_{n = 1}^{+\infty} l\de(n) \1_{\{ t \ge T\de(n) \}}.
\end{equation}
Let $N\de(t)$ be the number of $n$'s such that $T\de(n)$ falls inside $[0,t]$~:
$$
N\de(t) = \sum_{n = 1}^{+\infty} \1_{\{ t \ge T\de(n) \}}.
$$
Let $N$ be a positive integer, and $u$ a positive real number. From (\ref{upperboundLde}), we have the following upper bound on the tail distribution of $L\de(t)$~:
\begin{equation}
\label{boundtailL}
\ov{\P}[L\de(t) \ge \|h\|_\infty N u] \le \ov{\P}[\exists n \le N : l\de(n) \ge u] + \ov{\P}[N\de(t) > N].
\end{equation}
The first term of the sum is bounded by
\begin{equation}
\label{boundtailL1}
\sum_{n = 1}^{N} \ov{\P}[l\de(n) \ge u].
\end{equation}
Moreover, the random variable $l\de(n)$ either is equal to $0$ if the trap $x\de(n)$ is not actually visited, or is an exponential random variable which parameter is the inverse of the Green function at $x\de(n)$. We know from Proposition \ref{compG} of the Appendix that the Green function is uniformly bounded by a constant, say $C$, hence $l\de(n)$ is stochastically dominated by an exponential random variable of parameter $C^{-1}$ (uniformly in $n$ and in $\eps$). As a consequence, the sum in (\ref{boundtailL1}) is bounded by
$$
N e^{-u/C}.
$$
Let us know examine the rightmost term in (\ref{boundtailL}). We recall that the sequence of sites discovered by the random walk up to time $\eps^{-1} t$ is $(x_i)_{i \le r(\eps^{-1} t)}$. Let $B\de(i)$ be the indicator of the event that the site $x_i$ is a deep trap~:
$$
B\de(i) = \1_{\{ \eps^{1/\alpha} \tau_{x_i} \ge \delta \}}.
$$
Then one can rewrite $N\de(t)$ as
$$
N\de(t) = \sum_{i = 1}^{r(\eps^{-1} t)} B\de(i),
$$
which enables us to decompose the rightmost term in (\ref{boundtailL}) as~:
\begin{equation}
\label{boundtail3}
\ov{\P}\left[ \sum_{i = 1}^{r(\eps^{-1} t)} B\de(i) > N \right] \le \ov{\P}\left[ \sum_{i = 1}^{\eps^{-1} I } B\de(i) > N \right] + \ov{\P} [\eps r(\eps^{-1} t) > I],
\end{equation}
where $I$ is any positive integer. We begin by bounding the first term of this sum.
From Proposition~\ref{tauiid}, we know that $(B\de(i))_{i \in \N}$ forms a family of independent Bernoulli random variables of parameter
$$
\P[\eps^{1/\alpha} \tau_0 \ge \delta].
$$
According to (\ref{regvar}), this quantity is equivalent to $\eps \delta^{-\alpha}$ as $\eps$ tends to $0$. It is therefore smaller than $c_0 \eps$ for some large enough $c_0$, uniformly over $\eps$. We obtain, using Chebychev inequality~:
$$
\ov{\P}\left[ \sum_{i = 1}^{\eps^{-1} I} B\de(i) \ge N \right] \le e^{-N} \ov{\E}[\exp(B\de(1))]^{\eps^{-1} I}.
$$
Using the fact that $\ov{\E}[\exp(B\de(1))] \le 1 + c_0 \eps (e-1)$, we can bound the former by
$$
\exp\left( -N + \eps^{-1} I \ln(1 + \eps c_0 (e-1))  \right) \le \exp\left( -N +    I c_0 (e-1) \right).
$$
Choosing $I = c_1 N$ with $c_1 > 0$ small enough, this quantity decays exponentially fast as $N$ goes to infinity. We now turn to the second term on the right hand side of (\ref{boundtail3}), keeping $I = c_1 N$.
$$
\ov{\P} [\eps r(\eps^{-1} t) > c_1 N] \le \frac{\ov{\E}[(\eps r(\eps^{-1} t))^2]}{(c_1 N)^2},
$$
and Proposition~\ref{discovered} ensures that the numerator is uniformly bounded as $\eps$ varies. 

We have thus shown that there exists $C > 0$ such that, for any $\eps > 0$, one has~:
\begin{equation}
\label{endoftail}
\ov{\P}[L\de(t) \ge N u] \le N e^{-u/C} + e^{-N/C} + \frac{C}{N^2}.
\end{equation}
From this control of the tail of $L\de(t)$, one can check that
$$
\sup_{\eps > 0} \ov{\E}\left[\left( L\de(t) \right)^{3/2}\right]
$$
is finite (choosing for instance $u = N^{1/5}$ in (\ref{endoftail})), and this is a sufficient condition to ensure uniform integrability.
\end{proof}
\noindent \textbf{Remark.} From the relations (\ref{charactenv}) and (\ref{charactbis}), one obtains that
$$
c = \E\left[\ov{G}(\tau)^{-1}\right] = \left(\ov{\E}\left[\ov{G}(\tau_\delta(1))\right]\right)^{-1}.
$$
%
%
%
%
%
%

\section{Identification of the limit}
\label{s:identif}
\setcounter{equation}{0}
In this section, we begin by proving that $H\de$ converges in distribution as $\eps$ tends to $0$, and describe the limit subordinator in terms of its Laplace transform. Then, by an interversion of limits, we obtain the convergence of the law of $H\e$ as $\eps$ tends to $0$. We start with a summary of previous results.

\begin{prop}
\label{trioconv}
Let $(\eps_k)$ be a sequence such that (\ref{convbigjoint}) holds. The joint distribution of $(\tau_\delta^\circ(1),e_\delta(1),\tau_\delta(1))$ does not depend on the sequence $(\eps_k)$, and is described as follows~: the three components are independent, and their respective distributions are given by Propositions~\ref{taulim}, \ref{visittrap} and \ref{convenv}.
\end{prop}
\begin{proof}
It is a consequence of the above mentioned Propositions, together with Proposition~\ref{trioindep}.
\end{proof}
\noindent \textbf{Remark.} From this result, one could show that the random variables 
$$
\left(\eps^{1/\alpha} \tau_{x\de(1)} ,e\de(1), \tau\de(1) \right)
$$ 
jointly converge in law as $\eps$ goes to $0$.

We insist that, from now on, the law of $(\tau_\delta^\circ(1),e_\delta(1),\tau_\delta(1))$ may be considered without any mention of a particular sequence $(\eps_k)$. 
\begin{prop}
\label{cvHde}
For any $\delta > 0$, the law of $H\de$ under $\ov{\P}$ converges, for the $M_1$ topology and as as $\eps$ tends to $0$, to the law of a subordinator with Laplace exponent~:
\begin{equation}
\label{psideltaeq}
\psi_\delta(\lambda) = {c \delta^{-\alpha}} \ov{\E}[1 - e^{- \lambda e_\delta(1) \ov{G}(\tau_\delta(1)) \tau^\circ_\delta(1)}].
\end{equation}
\end{prop}
\begin{proof}
It is sufficient to show that, for any given sequence that converges to $0$, one can extract a further subsequence $(\eps_k)_{k \in \N}$ along which the law of $H\de$ converges to the law of a subordinator, whose Laplace exponent $\psi_\delta$ satisfies (\ref{psideltaeq}). 

Let us give ourselves a sequence that converges to $0$. Because the random variables in (\ref{bigjoint}) are tight, one can extract a further subsequence $(\eps_k)_{k \in \N}$ for which (\ref{convbigjoint}) holds.

Proposition~\ref{limitHform} states that, as $k$ goes to infinity, the law of the process $H\dek$ converges to the law of the process $H_\delta$ defined in (\ref{convH}). Moreover, we know from Proposition~\ref{indepincr} that $H_\delta$ is a subordinator. We can therefore define its Laplace exponent, say $\psi_\delta$, which satisfies, for any $\lambda, t \ge 0$~:
\begin{equation}
\label{eq:psi1}
\ov{\E}[e^{- \lambda H_\delta(t)}] = e^{-t \psi_\delta(\lambda)}.
\end{equation}

We recall that, because $H_\delta$ is a subordinator, the height and the instant of occurence of the first jump are independent random variables. Decomposing according to whether a first jump occurs or not (and using Proposition~\ref{poisson}), one can see that
\begin{equation}
\label{eq:psi2}
\ov{\E}[e^{- \lambda H_\delta(t)}] = 1-\frac{c t}{\delta^\alpha} + \frac{c t}{\delta^\alpha} \ov{\E}\left[\exp\left(- \lambda e_\delta(1) \ov{G}(\tau_\delta(1)) \tau^\circ_\delta(1) \right)\right] + O(t^2).
\end{equation}
According to (\ref{eq:psi1}), it is also equal to
$$
e^{-t \psi_\delta(\lambda)} = 1-t \psi_\delta(\lambda) + O(t^2),
$$
which, when compared with (\ref{eq:psi2}), proves the announced result.
\end{proof}

\noindent \textbf{Remark.} Similarly, one obtains that the law of $L\de$ under $\ov{\P}$ converges, as $\eps$ tends to $0$, to a subordinator with Laplace exponent 
$$
c \delta^{-\alpha} \ov{\E}[ 1 - e^{- \lambda e_\delta(1) (\ov{G}h)(\tau_\delta(1))}].
$$

From now on, the law of the process $H_\delta$ is well defined, independently of any particular sequence $(\eps_k)$~: it is the law of a subordinator whose Laplace exponent is $\psi_\delta$.

\begin{prop}
\label{p:diagram}
Possibly enlarging the probability space, there exists a process $H$ such that the following diagram holds~:
\begin{displaymath}
\label{diagram}
\begin{array}{cccc}
H\de & \xrightarrow[\eps \to 0]{} &  H_\delta & \\
\downarrow &  & \downarrow & (\delta \to 0) \\
H^{(\eps)} & \xrightarrow[\eps \to 0]{} &  H, &
\end{array}
\end{displaymath}
where arrows represent convergence in distribution under $\ov{\P}$ for the $M_1$ topology. 
\end{prop}
Before proving the Proposition, let us define the space $D_{\uparrow}([0,t], \R)$ of cadlag increasing processes from $[0,t]$ to $\R$ and with value $0$ at $0$. We recall a characterization of tightness of probability measures on $D_{\uparrow}([0,t], \R)$ \cite[Theorem 12.12.3]{whitt}.
\begin{lem}
\label{tightness}
Let $(h_n)_{n \in \N}$ be random (with respect to the measure $\ov{\P}$) elements of $D_{\uparrow}([0,t], \R)$. The family of distributions of $h_n$ is tight for the $M_1$ topology if and only if the following three properties hold~:
\begin{equation}
\label{tightness1}
\forall \eta > 0 \ \exists C > 0 \ \forall n \ : \ \ov{\P}[h_n(t) \ge C] \le \eta,
\end{equation}
\begin{equation}
\label{tightness2}
\forall \eta, \eta' > 0 \ \exists \iota > 0 \ \forall n \ : \  \ov{\P}[h_n(\iota) \ge \eta'] \le \eta,
\end{equation}
\begin{equation}
\label{tightness3}
\forall \eta, \eta' > 0 \ \exists \iota > 0 \ \forall n \ : \ \ \ov{\P}[h_n(t) - h_n(t-\iota) \ge \eta'] \le \eta .
\end{equation}
\end{lem}
\begin{proof}
It is a simple rewriting of \cite[Theorem 12.12.3]{whitt}, using the fact that we restrict here our attention to increasing processes with value $0$ at $0$. 
\end{proof}
\begin{proof}[Proof of Proposition~\ref{p:diagram}]
We begin by showing that there exists $\ov{c} > 0$ such that, for any $\eps, \delta > 0$~:
\begin{equation}
\label{convunifHde}
\ov{\E}\left[\sup_{[0,t]} |H\de - H^{(\eps)}|\right] \le \ov{c} t \delta^{1-\alpha}.
\end{equation}

Observe that
\begin{equation}
\label{convunif1}
\ov{\E}\left[\sup_{[0,t]} |H^{(\eps)} - H\de|\right] = \eps^{1/\alpha} \int_0^{\eps^{-1}t} \ov{\E}\left[  \tau_{\hX_s} \1_{\{ \eps^{1/\alpha} \tau_{\hX_s} < \delta \}}  \right] \ \d s.
\end{equation}
The expectation in the integral is in fact independent of $s$, due to the stationarity of the environment viewed by the particle under $\ov{\E}$. Using Fubini's theorem, we can bound it the following way~:
\begin{eqnarray*}
\E[\tau_{0} \1_{\{ \eps^{1/\alpha} \tau_{0} < \delta \}}] & = & \int_{x=0}^{\eps^{-1/\alpha} \delta} \int_{y=0}^x \d y \ \d \mu_0(x) \\
& \le & \int_{y=0}^{ \eps^{-1/\alpha} \delta} \mu_0([y,+\infty)) \ \d y.
\end{eqnarray*}
Using our hypothesis (\ref{regvar}) concerning the tail behaviour of $\mu_0$, there exists $C > 0$ such that for any $x > 0$, one has
$$
\mu_0([x,+\infty)) \le \frac{C}{x^\alpha}.
$$
Integrating this estimate, and then coming back to (\ref{convunif1}), we obtain inequality (\ref{convunifHde}).

We can now show that the family of distributions of $H^{(\eps)}$ is tight for the $M_1$ topology, using Lemma~\ref{tightness}. Let us begin by checking condition~(\ref{tightness1}). We fix some $\delta > 0$, and observe that, for any $C > 0$~:
\begin{equation}
\label{Hetight}
\ov{\P}[H\e(t) \ge 2C] \le \ov{\P}[H\de(t) \ge C] + \ov{\P}[H\e(t) - H\de(t) \ge C].
\end{equation}
Let us now give ourselves $\eta > 0$. As the law of $H\de$ converges as $\eps$ tends to $0$, Lemma~\ref{tightness} ensures that, for a large enough $C$, one has, for any $\eps > 0$~:
$$
\ov{\P}[H\de(t) \ge C] \le \eta.
$$
The second term of the sum in (\ref{Hetight}) is bounded by $\ov{c} t \delta^{1-\alpha}/C$. Possibly enlarging $C$, this term can be made smaller than $\eta$ as well, and condition~(\ref{tightness1}) is thus proved. Conditions~(\ref{tightness2}) and~(\ref{tightness3}) are obtained the same way.

We now show that there is in fact a unique possible limit law for $H^{(\eps)}$. Let $(\eps_k)_{k \in \N}$ be a sequence decreasing to $0$ and such that the law of $H^{(\eps_k)}$ converges to the law of some process $H$. First, one can easily check that the $M_1$ distance \cite[(3.3.4)]{whitt} is dominated by the supremum distance. Inequality (\ref{convunifHde}) thus guarantees that the convergence of $H\dek$ towards $H^{(\eps_k)}$ is uniform in $k$, and one can intervert limits \cite[Theorem~4.2]{bill}~: the law of $H$ is also the limit of the law of $H_\delta$ as $\delta$ tends to $0$. In particular, the law of $H$ does not depend on the sequence $(\eps_k)$. 

As we verified that $H\e$ is tight and has a unique possible limit law, and also that the diagram (\ref{diagram}) holds, the proposition is proved.
\end{proof}

\begin{prop}
\label{cvH}
The law of $H$ is that of an $\alpha$-stable subordinator, whose Laplace exponent is given by~:
\begin{equation}
\label{defpsi}
\psi(\lambda) = \Gamma(\alpha + 1) \E\left[ \ov{G}(\tau)^{\alpha - 1} \right] \int_{0}^{+ \infty} (1 - e^{-\lambda u}) \frac{\alpha}{u^{\alpha + 1}}  \ \d u,
\end{equation}
where $\Gamma$ is Euler's Gamma function.
\end{prop}
\begin{proof}
We begin by showing that the Laplace exponent $\psi_\delta(\lambda)$ of $H_\delta$ converges, for any $\lambda \ge 0$, to $\psi(\lambda)$ defined in (\ref{defpsi}). Let $\nu$ be the law of $e_\delta(1) \ov{G}(\tau_\delta(1))$. We recall from Proposition~\ref{trioconv} that the joint law of $(\tau_\delta^\circ(1),e_\delta(1),\tau_\delta(1))$ is known. As a consequence, one can check that the measure $\nu$ does not depend on $\delta$. From Proposition~\ref{cvHde}, we obtain that
$$
\psi_\delta(\lambda) = c \delta^{-\alpha} \int_{x \ge \delta} (1-e^{- \lambda x v}) \delta^\alpha  \frac{\alpha}{x^{\alpha + 1}} \ \d x \d \nu(v)  .
$$
The terms $\delta^\alpha$ cancel out, and the change of variables $u = x v$ leads to
\begin{equation}
\label{psidlim}
\psi_\delta(\lambda) = c \int_{u = 0}^{+ \infty} (1 - e^{-\lambda u}) \frac{\alpha}{u^{\alpha + 1}} \int_{v = 0}^{u/\delta} v^\alpha \ \d \nu(v) \ \d u.
\end{equation}
Moreover, one has that
$$
\int_{0}^{u/\delta} v^\alpha \ \d \nu(v) \xrightarrow[\delta \to 0]{} \int_{0}^{+\infty} v^\alpha \ \d \nu(v),
$$
and, using the description of $\nu$ provided by Proposition~\ref{trioconv}~:
$$
\int_{0}^{+\infty} v^\alpha \ \d \nu(v) = c^{-1} \ \E[\ov{G}(\tau)^{\alpha - 1}] \int_0^{+\infty} x^\alpha e^{-x} \ \d x,
$$
the last integral being equal to $\Gamma(\alpha + 1)$. From equation~(\ref{psidlim}) and using monotone convergence theorem, we obtain that $\psi_\delta(\lambda)$ converges to $\psi(\lambda)$ as $\delta$ tends to $0$.

It remains to check that $H$ is a subordinator, and that $\psi$ is its Laplace exponent. Some caution is necessary due to the fact that convergence for Skorokhod's $M_1$ topology does not imply convergence of all finite dimensional distributions in general. However, it is clear from the argument at the beginning of the proof of Proposition~\ref{limitHform} that convergence of finite-dimensional distributions holds whenever the times considered do not belong to the set 
$$
T^\circ = \{ t \in \R_+ : \ov{\P}[H(t) \neq H(t^-)] > 0 \}.
$$
This set is countable, as \cite[Section 15]{bill} shows. Hence, for any $\lambda_1,\ldots,\lambda_n \ge 0$, and any $t_1 \le \cdots \le t_n$ outside $T^\circ$, one has 
\begin{equation*}
\begin{split}
& \ov{\E}[e^{- \lambda_1 H(t_1) - \lambda_2 (H(t_2)-H(t_1)) - \cdots - \lambda_n (H(t_n)-H(t_{n-1}))}] \\
& \qquad = \lim_{\delta \to 0} \ov{\E}[e^{- \lambda_1 H_\delta(t_1) - \lambda_2 (H_\delta(t_2)-H_\delta(t_1)) - \cdots - \lambda_n (H_\delta(t_n)-H_\delta(t_{n-1}))}] \\
& \qquad = e^{-t_1 \psi(\lambda_1) - (t_2-t_1) \psi(\lambda_2) - \cdots - (t_n - t_{n-1}) \psi(\lambda_n)}.
\end{split}
\end{equation*}
Finally, right continuity of the process $H$ ensures that the above equality holds in fact for every $t_1,\ldots,t_n$, thus finishing the proof.
\end{proof}
%
%
%
%
%
%
\section{Joint convergence}
\label{s:joint}
\setcounter{equation}{0}
In this section we will identify the limit of the joint distribution of $(\hX\e, H^{(\eps)})$ under the annealed measure $\ov{\P}$. 

The first step is to describe the limit law of $\hX\e$. We state it directly in its quenched form, although in this section, the annealed version would be sufficient.

\begin{prop}
\label{tclquenched}
For almost every $\tau$, the law of $\hX\e$ under $\PPt_0$ converges, for the $J_1$ topology and as $\eps$ tends to $0$, to the law of a non-degenerate Brownian motion~$B$. 
\end{prop}
\begin{proof}
We refer to \cite[Theorem 1.1]{bardeu} for a proof of this fact.
\end{proof}

\begin{prop}
\label{cvjointannealed}
The law of $(\hX\e,H\e)$ under $\ov{\P}$ converges, for the $J_1 \times M_1$ topology and as $\eps$ tends to $0$, to the law of two independent processes $(B,H)$, where $B$ and $H$ are the processes appearing respectively in Propositions~\ref{tclquenched} and~\ref{cvH}.
\end{prop}
\begin{proof}
We write $Z\e$ for $(\hX\e, H\e)$. Propositions~\ref{tclquenched} and \ref{cvH} ensure the convergence in distribution of the two marginals of $Z\e$. In particular, the law of $Z\e$ is tight. Let $(\eps_k)$ be a sequence such that the law of $Z\ek$ under $\ov{\P}$ converges, and let us write $Z = (B,H)$ for the limit. The distributions of $B$ and of $H$ are known, and what we need to show is that these random variables are independent.

First, it is clear that from the convergence in the product $J_1 \times M_1$ topology, one can deduce the convergence of the finite dimensional distributions of $Z\e$, following the argument given at the beginning of the proof of Proposition~\ref{limitHform}. Then, one can follow the proof of Proposition~\ref{indepincr}, replacing Laplace transform by Fourier transform for definiteness, and obtain that the limit $Z$ is a L\'evy process. 

It follows from the L\'evy-Khintchine decomposition that the L\'evy process $Z$ can be decomposed into $Z^{(1)} + Z^{(2)}$, where $Z^{(1)}$ is a continuous process, $Z^{(2)}$ is pure jump, and $Z^{(1)},Z^{(2)}$ are independent \cite[Section I.1]{bertoin}. The decomposition into the sum of a continuous process and a pure jump one being unique, it follows that $Z^{(1)} = (B,0)$ and $Z^{(2)} = (0,H)$, which proves the proposition.
\end{proof}
%
%
%
%
%
%
\section{From annealed to quenched}
\label{s:quenched}
\setcounter{equation}{0}
From the knowledge of the convergence of $Z\e = (\hX\e,H\e)$ towards $Z = (B,H)$ under the annealed law $\P\PPt_0$, we would like to obtain convergence under $\PPt_0$ for almost every~$\tau$. This can be obtained by a kind of concentration argument that is due to \cite{bs}, and consists in checking that the variance of certain functionals of $Z\e$ decays sufficiently fast when $\eps$ tends to $0$ (a polynomial decay being sufficient). 

As a first step, we consider the joint law of increments of $Z$ on intervals that do not contain~$0$. In other words, for some $0 < t_0 \le \cdots \le t_{n}$, we consider the law of 
\begin{equation}
\label{incremout0}
\left( Z\e(t_1) - Z\e(t_0),\ldots, Z\e(t_{n}) - Z\e(t_{n-1}) \right).
\end{equation}
Using \cite[Lemma~4.1]{bs} together with Theorem~\ref{gtrajtau}, we will see that, for almost every environment, the law of increments of the form (\ref{incremout0}) under $\PPt_0$ converges to the law of the increments of $Z$.

This statement concerning the law of increments of the form (\ref{incremout0}) is weaker than the convergence of all finite-dimensional distributions, but is still sufficient if one can prove the tightness of $Z\e$. We can borrow the tightness of $\hX\e$ from Proposition~\ref{tclquenched}. In order to prove the tightness of $H\e$, we will in fact prove the convergence of its finite-dimensional distributions (which is a sufficient condition, see Lemma~\ref{tightness}). As we pointed out, it is not enough for this purpose to control the distributions of increments of $H\e$ on intervals that do not contain $0$, so we will need additional information concerning the behaviour of $H\e$ for small times. 

We start by giving this necessary control of $H\e$ for small times.

\begin{prop}
\label{Hpres0}
For any $\nu > 0$ and any $\gamma < \nu/\alpha$, the probability
$$
\ov{\P}[H\e(\eps^\nu) > \eps^\gamma]
$$
decays polynomially fast to $0$ as $\eps$ tends to $0$.
\end{prop}
\begin{proof}
It is in fact sufficient to show that, for any $\beta > 0$, the probability
\begin{equation}
\label{Hpres0equiv}
\ov{\P}[H\e(1) > \eps^{-\beta}]
\end{equation}
decays polynomially fast to $0$ as $\eps$ tends to $0$, as one can check using the fact that $H\e(\eps^\nu) = \eps^{\nu/\alpha} H^{(\eps^\nu)}(1)$.

Up to time $\eps^{-1}$, the random walk $\hX$ discovers $r(\eps^{-1})$ sites. Writing $l_i$ for the total time spent by the random walk on the $i^{\text{th}}$ discovered site
$$
l_i = \int_0^{+\infty} \1_{\{\hX_s = x_i\}} \ \d s,
$$
we can bound $H\e(1)$ by
$$
\eps^{1/\alpha} \sum_{i=1}^{r(\eps^{-1})} l_i \tau_{x_i},
$$
where $(x_i)$ is the exploration process defined in (\ref{defxn}). For any $N$, we thus have
$$
\ov{\P}[H\e(1) > \eps^{-\beta}] \le \ov{\P}\left[ \eps^{1/\alpha} \sum_{i=1}^{\eps^{-1} N} l_i \tau_{x_i} > \eps^{-\beta} \right] + \ov{\P}[ r(\eps^{-1}) > \eps^{-1} N ].
$$
Because of Proposition~\ref{discovered}, the second term is bounded by $C N^{-2}$, uniformly over $\eps$. In order to ensure polynomial decay, we choose $N$ as a small negative power of $\eps$, say $\eps^{-\gamma}$ for some $\gamma > 0$ to be fixed. With this choice of $N$, the first term becomes
$$
\ov{\P}\left[ \eps^{1/\alpha} \sum_{i=1}^{\eps^{-1-\gamma}} l_i \tau_{x_i} > \eps^{-\beta} \right]. 
$$
We choose another small parameter $\gamma'$, and decompose the above probability as
$$
\ov{\P}[ \exists i \le \eps^{-1-\gamma} :  l_i > \eps^{-\gamma'} ] + \ov{\P}\left[ \eps^{1/\alpha - \gamma'} \sum_{i=1}^{\eps^{-1-\gamma}} \tau_{x_i} > \eps^{-\beta}\right].
$$
The random variable $l_i$ is an exponential random variable, and moreover, its mean value, which is the Green function at $x_i$, is bounded by some constant as one can see from Proposition~\ref{compG}. Hence, the first term of the sum above is bounded by 
$$
\sum_{i=1}^{\eps^{-1-\gamma}} \ov{\P}[l_i > \eps^{-\gamma'}] \le \eps^{-1-\gamma} e^{-\eps^{-\gamma'}/C},
$$
which converges to $0$ faster than any polynomial. There remains to check that
$$
\ov{\P}\left[ \eps^{1/\alpha - \gamma'} \sum_{i=1}^{\eps^{-1-\gamma}} \tau_{x_i} > \eps^{-\beta} \right]
$$
converges polynomially fast to $0$. We know from Proposition~\ref{tauiid} that under $\ov{\P}$, the random variables $(\tau_{x_i})$ are independent and identically distributed according to $\mu_0$. Hence, because of the tail behaviour (\ref{regvar}), the sum of $\tau_{x_i}$ appearing above is of order $\eps^{-(1+\gamma)/\alpha}$, and a natural condition for this polynomial decay to hold seems to be that $\gamma'+\gamma/\alpha < \beta$. This condition is shown to be sufficient in \cite[Theorem~3]{BK} (note that there is a misprint in condition (d) of this theorem, where the sign $\Sigma$ should be replaced by the sign $E$).
\end{proof}

We will now proceed to prove that, for almost every $\tau$, the law of $H\e$ converges under $\PPt_0$, although our only true concern for now is that of tightness. 

As we said before, the argument of \cite[Lemma~4.1]{bs} requires the decay of the variance of certain functionals of $H\e$. Let $\lambda_1,\ldots,\lambda_n \ge 0$, and $0 < t_1 < \cdots < t_n$. For any increasing process $h$, we define $F(h)$ as
\begin{equation}
\label{Ftype}
F(h) = \exp\big( -\lambda_1 h(t_1) - \lambda_2 (h(t_2) - h(t_{1})) - \cdots - \lambda_n (h(t_n) - h(t_{n-1})) \big).
\end{equation}
\begin{prop}
\label{concentrpoly}
For $F$ defined by (\ref{Ftype}) and $d \ge 5$, the variance of $\EEt_0[F(H\e)]$ converges to $0$ polynomially fast as $\eps$ tends to $0$. 
\end{prop}
\begin{proof}
Let $\nu \in (0,1)$. We define
$$
P\e = \exp\left( -\lambda_1 (H\e(t_1) - H\e(\eps^\nu))  - \cdots - \lambda_n (H\e(t_n) - H\e(t_{n-1})) \right),
$$
which enables us to decompose $F(H\e)$ as
\begin{equation}
\label{FHdecomp}
F(H\e) = e^{-\lambda_1 H\e(\eps^\nu)}  P\e.
\end{equation}
We momentarily admit the following lemma.
\begin{lem}
\label{l:concentrpoly}
If $d \ge 5$ and $\nu < 1/5$, then the variance of $\EEt_0[P\e]$ converges to $0$ polynomially fast as $\eps$ tends to $0$.
\end{lem}
Let us see how to finish the proof of Proposition~\ref{concentrpoly}, choosing some $\nu < 1/5$ (and $d \ge 5$). We will show that the variances of $\EEt_0[P\e]$ and $\EEt_0[F(H\e)]$ are close enough to conclude. Note that, from the decomposition~(\ref{FHdecomp}), one has
$$
0 \le P\e - F(H\e)  \le 1 - e^{-\lambda_1 H\e(\eps^\nu)}.
$$
It readily follows that
\begin{equation}
\label{boundPeps}
0 \le \E\left[ \EEt_0[P\e] \right] - \E\left[ \EEt_0[F(H\e)] \right]  \le 1 - \ov{\E}\left[ e^{-\lambda_1 H\e(\eps^\nu)} \right].
\end{equation}
It follows from Proposition~\ref{Hpres0} that the term on the right hand side converges to~$0$ polynomially fast, as $\eps$ tends to $0$. Similarly, we have
$$
0 \le \EEt_0[P\e]^2 - \EEt_0[F(H\e)]^2  \le 2 \left( \EEt_0[P\e] - \EEt_0[F(H\e)] \right).
$$
Integrating this inequality, and using the upper bound from~(\ref{boundPeps}), we obtain that the difference
$$
\E\left[ \EEt_0[P\e]^2 \right] - \E\left[ \EEt_0[F(H\e)]^2 \right]
$$
also converges polynomially fast to $0$, as $\eps$ tends to $0$. As a consequence, the difference between the variances of $\EEt_0[P\e]$ and $\EEt_0[F(H\e)]$ converges to $0$ polynomially fast, and Proposition~\ref{concentrpoly} is obtained using Lemma~\ref{l:concentrpoly}.
\end{proof}
\begin{proof}[Proof of Lemma~\ref{l:concentrpoly}]
We define the function $g\e(h)$ as
\begin{equation*}
g\e(h) = \exp\left(  -\lambda_1 (h(t_1-\eps^\nu) - h(0)) - \cdots - \lambda_n (h(t_n-\eps^\nu) - h(t_{n-1}-\eps^\nu) \right),
\end{equation*}
and we let $f(\tau) = \EEt_0[g\e(H\e)]$. Then $g\e(H\e)$ depends only on the trajectory up to time $\eps^{-1} (t_{n} - \eps^\nu) \le \eps^{-1} t_n$, and is translation invariant. As given by (\ref{vareq72}), one can rewrite $P\e$ as
$$
P\e = \EEt_0[f(\htau(\eps^{-1} \eps^\nu))] = f_{\eps^{\nu-1}}(\tau).
$$
As we assume that $d \ge 5$, Theorem~\ref{gtrajtau} shows that $\var(P\e) = \var(f_{\eps^{\nu-1}})$ is bounded by a constant times $\eps^{(1-\nu) d/2 - 2}$, so it is enough to chose $\nu < 1/5$ to guarantee a polynomial decay of the variance.
\end{proof}

We can now derive, following the method of proof of \cite[Lemma 4.1]{bs}, the convergence of the law of $H\e$ in the quenched sense.
\begin{prop}
\label{cvHquenched}
For almost every $\tau$, the law of $H\e$ under $\PPt_0$ converges, for the $M_1$ topology and as $\eps$ tends to $0$, to the law of $H$. 
\end{prop}
\begin{proof}
We know from Proposition~\ref{p:diagram} that $H\e$ converges to $H$ under the measure $\ov{\P}$ for the $M_1$ topology. As we saw before, this convergence, together with the knowledge that the limit described in Proposition~\ref{cvH} has no deterministic times with positive probability of jump, implies convergence of finite dimensional distributions under the annealed measure. For $F$ defined by (\ref{Ftype}), we thus have
\begin{equation}
\label{convergeanneal}
\E\EEt_0[F(H\e)] \xrightarrow[\eps \to 0]{} \ov{\E}[F(H)] = \exp(-t_1 \psi(\lambda_1) - \cdots - (t_n-t_{n-1}) \psi(\lambda_n)),
\end{equation}
where $\psi$ is the Laplace exponent of $H$ defined in (\ref{defpsi}).

Moreover, we have seen in Proposition~\ref{concentrpoly} that $\var(\EEt_0[F(H\e)])$ decays to $0$ polynomially fast. Let $\mu \in (0,1)$. We thus have that
$$
\sum_{n = 1}^{+ \infty} \var(\EEt_0[F(H^{(\mu^n)})]) < + \infty.
$$
As a consequence, the convergence of
$$
\EEt_0[F(H^{(\mu^n)})]
$$
towards $\ov{\E}[F(H)]$ holds almost surely. In fact, with probability one, this convergence holds jointly for any function $F$ of the form (\ref{Ftype}) with $\lambda_1,\ldots,\lambda_n,t_1,\ldots,t_n$ and $\mu$ rationals. Using the monotonicity of $H\e$ and the continuity of the limit (see (\ref{convergeanneal})), the convergence can be extended to any $\lambda_1,\ldots,\lambda_n,t_1,\ldots,t_n$ simultaneously.

On the set of full measure where this joint convergence holds, we will show that for any $F$ of the form (\ref{Ftype}), one has
\begin{equation}
\label{cvfdquench}
\EEt_0[F(H\e)] \xrightarrow[\eps \to 0]{} \ov{\E}[F(H)] . 
\end{equation}
In other words, we will show that for any $\tau$ belonging to this set of full measure, the finite-dimensional distributions of $H\e$ converge to those of $H$. In order to do so, we approximate $H\e$ by some $H^{(\mu^n)}$, for a well chosen $n$. Let $n_\eps$ be the smallest integer satisfying $\mu^{n_\eps} < \eps$. The function $F$ defined in (\ref{Ftype}) is such that, for any two increasing processes $h$ and $h'$ starting from $0$~:
$$
|F(h) - F(h')| \le C \max_{1 \le i \le n} |h(t_i) - h'(t_{i})| \wedge 1.
$$
Observe that 
$$
H\e(t_i) = \left( \frac{\eps}{\mu^{n_\eps}} \right)^{1/\alpha} H^{(\mu^{n_\eps})}\left( \frac{\mu^{n_\eps}}{\eps} t_i \right),
$$
and moreover, because of the definition of $n_\eps$ (and the monotonicity of $H^{(\mu^{n_\eps})}$), the latter is greater than $H^{(\mu^{n_\eps})}(\mu t_i)$, and as a consequence,
\begin{equation}
\label{compHmuH}
0 \le H^{(\mu^{n_\eps})}(t_i) - H\e(t_i)  \le  H^{(\mu^{n_\eps})}(t_i) - H^{(\mu^{n_\eps})}(\mu t_i)
\end{equation}

The quantity
$$
\limsup_{\eps \to 0} \left| \EEt_0[F(H\e)] - \EEt_0[F(H^{(\mu^{n_\eps})})] \right|
$$
is thus, up to a constant, bounded by
$$
\limsup_{\eps \to 0}  \EEt_0\left[\max_{1 \le i \le n} |H\e(t_i) - H^{(\mu^{n_\eps})}(t_{i})| \wedge 1 \right],
$$
which, as we obtain from the inequalities (\ref{compHmuH}), is bounded by 
\begin{multline*}
\limsup_{\eps \to 0}  \EEt_0\left[\max_{1 \le i \le n} (H^{(\mu^{n_\eps})}(t_i) - H^{(\mu^{n_\eps})}(\mu t_i)) \wedge 1 \right] 
\\ =  \ov{\E}\left[\max_{1 \le i \le n} (H(t_i) - H(\mu t_{i})) \wedge 1 \right].
\end{multline*}
The process $H$ being almost surely continuous at deterministic times, this last quantity tends to $0$ as $\mu$ converges to $1$. We thus obtain the claim (\ref{cvfdquench}), letting $\mu$ tend to $1$ along rationals. 

What is left to do is to check the tightness of the process in the sense of the $M_1$ topology. Lemma~\ref{tightness} shows that, as far as increasing processes are concerned, convergence of the finite-dimensional distributions is sufficient. 
\end{proof}

We can now prove our main result, namely the almost sure convergence of the joint process $(\hX\e,H\e)$. We recall from Proposition~\ref{cvjointannealed} that the process $(B,H)$ is such that $B$ is the Brownian motion appearing in Proposition~\ref{tclquenched}, $H$ is the subordinator whose Laplace exponent is given in Proposition~\ref{cvH}, and the random variables $B$, $H$ are independent.
\begin{prop}
\label{cvjointquenched}
For almost every $\tau$, the law of $(\hX\e,H\e)$ under $\PPt_0$ converges, for the $J_1 \times M_1$ topology and as $\eps$ tends to $0$, to the law of $(B,H)$.
\end{prop}
\begin{proof}
We recall that we write $Z\e$ for the process $(\hX\e, H\e)$, and $Z$ for the process $(B,H)$. As a first step, Propositions~\ref{tclquenched} and \ref{cvHquenched} ensure that, for $\tau$ in a set of full measure $\Omega_1$, the laws of $Z\e$ under $\PPt_0$ are tight. 

We now show that the laws of the increments of $Z\e$, on intervals that do not contain $0$, converges almost surely to those of $Z$. Let $\lambda_1,\ldots,\lambda_n \in \R^{d+1}$, and $0 < t_0 < \cdots < t_n$. For a process $z$ with values in $\R^{d+1}$, we define $G(z)$ as
\begin{equation}
\label{Gtype}
G(z) = \exp\big( i \lambda_1 \cdot (z(t_1)-z(t_0)) + \cdots + i \lambda_n \cdot (z(t_n)-z(t_{n-1})) \big).
\end{equation}

From Proposition~\ref{cvjointannealed}, we know that 
$$
\ov{\E}[G(Z\e)] \xrightarrow[\eps \to 0]{} \ov{\E}[G(Z)].
$$
Moreover, one can adapt the proof of Lemma~\ref{l:concentrpoly} to show that the variance of $\EEt_0[G(Z\e)]$ converges to $0$ polynomially fast as $\eps$ tends to $0$. Indeed, the main difference between $P\e$ and $G(Z\e)$ is that $\eps^\nu$ should be replaced by $t_0 > 0$. For any $\mu \in (0,1)$, the sum
$$
\sum_{n = 1}^{+ \infty} \var(\EEt_0[G(Z^{(\mu^n)})])
$$
is thus finite, and as a consequence, the convergence 
$$
\EEt_0[G(Z^{(\mu^n)})] \xrightarrow[n \to +\infty]{} \ov{\E}[G(Z)]
$$
holds almost surely. In fact, for $\tau$ in a set of full measure, say $\Omega_2$, this convergence holds for any function $G$ of the form (\ref{Gtype}) with $\lambda_1,\ldots,\lambda_n,t_0,\ldots,t_n,\mu$ rationals. We can then proceed as in the proof of Proposition~\ref{cvHquenched} to show that, for any such $G$ and for any $\tau \in \Omega_2$, one has
$$
\EEt_0[G(Z\e)] \xrightarrow[\eps \to 0]{} \ov{\E}[G(Z)].
$$

Let $\tau$ be an element of $\Omega_1 \cap \Omega_2$, and let $\eps_k$ be a sequence such that the law of $Z\ek$ under $\PPt_0$ converges to the law of some $\td{Z}$ (for convenience, we assume that it is defined on the same probability space equipped with the measure $\ov{\P}$). As $\tau$ belongs to $\Omega_2$, we know that for any function $G$ of the form (\ref{Gtype}) with $\lambda_1,\ldots,\lambda_n,t_0,\ldots,t_n$ rationals, one has
$$
\ov{\E}[G(Z)] = \ov{\E}[G(\td{Z})].
$$
Using right continuity of the processes, the equality extends to any $G$ with $0 \le t_0 \le \cdots \le t_n$. The Fourier transform being continuous, it holds as well for any $\lambda_1,\ldots,\lambda_n \in \R^{d+1}$, and thus $Z$ and $\td{Z}$ have the same law.

To summarize, we have shown that, for $\tau \in \Omega_1 \cap \Omega_2$, the laws of $Z\e$ are tight and have a unique possible limit point, namely $Z$. This proves the proposition, as the set $\Omega_1 \cap \Omega_2$ is of full measure.
\end{proof}

\noindent \textbf{Remark.} What we really used from Proposition~\ref{tclquenched} is the annealed invariance principle, and the tightness of $\hX\e$ under the quenched measure. One can also prove tightness directly, in a way similar to what we did here for the tightness of~$H\e$. However, one then needs some equivalent of Proposition~\ref{Hpres0} for $\hX$. Precisely, one needs to show that, for any $\beta > 1/2$, the probability
$$
\ov{\P}\left[ \sup_{s \le t} |\hX_s| \ge t^\beta \right]
$$
decays polynomially fast as $t$ tends to infinity. 
%
%
%
%
%
%
\section{Conclusion}
\label{s:concl}
\setcounter{equation}{0}
\begin{prop}
\label{concl}
For almost every $\tau$, the law of $X\e$ under $\PPt_0$ converges, for the $J_1$ topology and as $\eps$ tends to $0$, to the law of $B \circ H^{-1}$.
\end{prop}
\begin{proof}
Following the notation in \cite{whitt}, let us write $D_{\uparrow}$ (resp.~$D_{u,\uparrow \uparrow}$) for the subset of $D([0,+\infty),\R)$ made of increasing (resp.~unbounded and strictly increasing) functions, with value $0$ at $0$. We also let $C$ be the set of continuous functions from $[0,+\infty)$ to $\R$, equipped with the uniform topology, that we will write $U$. According to \cite[Corollary 13.6.4]{whitt}, the inverse map
$$
\left\{
\begin{array}{ccc}
D_{u,\uparrow \uparrow}, M_1 & \to & C,U \\
x & \mapsto & x^{-1}
\end{array}
\right.
$$
is continuous. As a consequence, $(H\e)^{-1}$ converges in distribution to $H^{-1}$ for the uniform topology, and a fortiori for the $J_1$ topology. Moreover, as we learn from \cite[Theorem 13.2.1]{whitt}, the composition map
$$
\left\{
\begin{array}{ccc}
D([0,+\infty),\R^d) \times D_{\uparrow}, J_1 \times J_1 & \to & D([0,+\infty),\R^d),J_1 \\
(x,y) & \mapsto & x \circ y
\end{array}
\right.
$$
is measurable, and continuous on pairs of continuous functions. Hence, $X\e = \hX\e \circ (H\e)^{-1}$ converges in distribution to $B \circ H^{-1}$, and the proposition is proved.
\end{proof}

%
%
%
%
%
%
\section{Appendix}
\label{s:trans}
\setcounter{equation}{0}
Let $\mfk{L}$ be the generator of the random walk $\hX$, defined by~:
$$
\mfk{L} f (x) = \sum_{y \sim x}  (\tau_x \tau_y)^a (f(y)-f(x)).
$$
We write $(\cdot,\cdot)$ for the scalar product with respect to the counting measure. We define the Dirichlet form associated to $\mfk{L}$, as
$$
\mfk{E}(f,f) = (-\mfk{L} f,f) = \frac{1}{2} \sum_{x,y \in \Z^d} (\tau_x \tau_y)^a (f(y)-f(x))^2,
$$
together with the Dirichlet form $\mfk{E}^\circ$ associated with the simple random walk, obtained by taking $a = 0$ in the expression above. Note that from the definition, as a consequence of our hypothesis that conductances are uniformly bounded from below by $1$, one Dirichlet form dominates the other~:
\begin{equation}
\label{compDirichlet}
\mfk{E}^\circ(f,f) \le \mfk{E}(f,f).
\end{equation}
Let $B_n = \{-n,\ldots, n\}^d$ be the box of size $n$, and $B_n'$ be its complement in $\Z^d$. We introduce the effective conductance $C_n$ between the origin and $B_n'$, which is given by the following variational formula~:
\begin{equation}
\label{defCn}
C_n(\tau) = \inf  \left\{\mfk{E}(f,f) \ | \ f(0) = 1, f_{|B_n'} = 0 \right\},
\end{equation}
and we let $C_n^\circ$ be defined the same way, with $\mfk{E}$ replaced by $\mfk{E}^\circ$. Furthermore, we define $\ov{C}_n(\tau)$ as
\begin{equation}
\label{defovCn}
\ov{C}_n(\tau) = \inf  \left\{\mfk{E}(f,f) \ | \ f_{\mcD(0)} = 1, f_{|B_n'} = 0 \right\},
\end{equation}
where we recall that $\mcD(0)$ is the set formed by the origin and its neighbours. It is intuitively clear that $\ov{C}_n(\tau)$ does not depend on $\tau_0$, and that $\ov{C}_n(\tau)$ is the limit as $\tau_0$ goes to infinity of $C_n(\tau)$. The next proposition provides a quantitative estimate on this convergence. We write $q^\tau(x,y)$ for the probability for the walk starting from $x$ to jump to the site $y$. 

\begin{prop}
\label{compC}
For any environment $\tau$, and any integer $n$, the following comparisons hold~:
$$
C_n^\circ \le C_n(\tau) \le \ov{C}_n(\tau),
$$
$$
\ov{C}_n(\tau) \le \left( \min_{y \sim 0} q^\tau(y,0) \right)^{-2} C_n(\tau).
$$
\end{prop}
\begin{proof}
The first two inequalities are obvious, using (\ref{compDirichlet}). Recall that we write $\mathcal{T}_0$ for the hitting time of $0$. Let $\mathcal{T}_{B_n'}$ be the hitting time of $B_n'$. There exists a unique function $f$ that minimizes (\ref{defCn}), which is given by
\begin{equation}
\label{minimizer}
f(x) = \PPt_x[\mathcal{T}_0 < \mathcal{T}_{B_n'}].
\end{equation}
Let us write $m$ for $\min_{y \sim 0} f(y)$, and consider the function
$$
g(x) = \min(m^{-1} f(x) ,1).
$$
Then $g$ is constant equal to $1$ on $\mcD(0)$, and is $0$ outside $B_n$. It is thus clear that
$$
\ov{C}_n(\tau) \le \mfk{E}(g,g).
$$
On the other hand, one has
$$
\mfk{E}(g,g) \le m^{-2} \mfk{E}(f,f) = m^{-2} C_n(\tau).
$$
The last claim of the Proposition follows from the observation that, for any $y$ neighbour of the origin,
$$
f(y) = \PPt_y[\mathcal{T}_0 < \mathcal{T}_{B_n'}] \ge q^\tau(y,0).
$$
\end{proof}

Recall the definition of $\sigma^\tau(x,y)$ from (\ref{defsigma}). If $y$ is a neighbour of $0$, one has~:
$$
q^\tau(y,0) = \frac{(\tau_0)^a}{\sum_{z \sim y} (\tau_z)^a} = \left( 1 + \frac{\sigma^\tau(y,0)}{(\tau_0)^a} \right)^{-1},
$$
from which it follows that $\ov{C}_n(\tau)$ is indeed the limit of $C_n(\tau)$ as $\tau_0$ tends to infinity.

For $n \in \N \cup \{\infty\}$, let $\mathcal{G}_n^\tau(\cdot,\cdot)$ (resp. $\mathcal{G}_n^\circ(\cdot,\cdot)$) be the Green function of the walk $\hX$ (resp. of the simple random walk) killed when exiting $B_n$, or without killing if $n = \infty$. The function $f$ in (\ref{minimizer}) that minimizes (\ref{defCn}) can be rewritten as
$$
f = \frac{\mathcal{G}_n^\tau(\cdot,0)}{\mathcal{G}_n^\tau(0,0)},
$$
and, as $-\mfk{L} \mathcal{G}_n^\tau(\cdot,0) = \1_{0}$ on $B_n$, we obtain~:
\begin{equation}
\label{CnG-1}
C_n(\tau) = \mfk{E}(f,f) = (-\mfk{L} f, f) = {\mathcal{G}_n^\tau(0,0)}^{-1}.
\end{equation}
We define $C_\infty^\circ$, $C_\infty(\tau)$ and $\ov{C}_\infty(\tau)$ as the limits of, respectively, $C_n^\circ$, $C_n(\tau)$ and $\ov{C}_n(\tau)$. Monotonicity ensures that these limits are well defined. Because of the transience of the simple random walk in dimension three and higher, we also know that $C_\infty^\circ$ is strictly positive (\cite[Theorem 2.3]{trees}), and thus $C_\infty(\tau)$ and $\ov{C}_\infty(\tau)$ as well. 

We recall that we write $G(\tau)$ for $\mathcal{G}_\infty^\tau(0,0)$, which is also $C_\infty(\tau)^{-1}$. We let $\ov{G}(\tau)$ be the inverse of $\ov{C}_\infty(\tau)$. In the next proposition, we will see that this definition coincides with the one given in~(\ref{limitG}).

\begin{prop}
\label{compG}
For any environment $\tau$, the following inequalities hold~:
$$
G(\tau) \le (C_\infty^\circ)^{-1},
$$
$$
\ov{G}(\tau) \le {G}(\tau) \le \left( \min_{y \sim 0} q^\tau(y,0) \right)^{-2} \ov{G}(\tau).
$$
In particular, $\ov{G}(\tau)$ satisfies~(\ref{limitG}).
\end{prop}
\begin{proof}
These are direct consequences of Proposition~\ref{compC}, together with the identity~(\ref{CnG-1}).
\end{proof}

Finally, we recall here a classical result concerning the decay of the transition probability of the random walk.

\begin{prop}
\label{diagonaldecay}
There exists $C > 0$ such that, for any $x,y \in \Z^d$ and any $t \ge 0$, one has~:
$$
\PPt_x[\hX_t = y] \le \frac{C}{t^{d/2}}.
$$
\end{prop}
\begin{proof}
Using \cite[Proposition 14.1]{woe} together with \cite[Corollary 4.12]{woe}, one knows that a Nash inequality holds for the simple random walk on $\Z^d$, in the sense that there exists $C > 0$ such that for any function $f$,
$$
\| f \|_2^{2+4/d} \le C_1 \mfk{E}^\circ(f,f) \| f \|_1^{4/d}.
$$
By~(\ref{compDirichlet}), the inequality is preserved if one changes $\mfk{E}^\circ$ by $\mfk{E}$. From the Nash inequality, one deduces the announced claim, following the argument of \cite{nash}, or equivalently \cite[Theorem 2.1]{cks}.
\end{proof}

\noindent \textbf{Acknowledgments.} The author would like to thank Pierre Mathieu for many insightful discussions about this work as well as detailed comments on earlier drafts.

\end{document}